\documentclass[a4paper,reqno]{amsart}

\usepackage{amsthm,amssymb,amsmath, amscd, epsfig, enumitem, verbatim, color, caption, faktor,esint, tikz-cd, xfrac, wrapfig}

\usepackage[utf8]{inputenc}
\usepackage[T1]{fontenc}
\usepackage[all]{xy}

\usepackage{hyperref}
\newcommand{\N}{\mathbb N}
\newcommand{\R}{\mathbb R}

\renewcommand{\H}{\mathbb H}

\newcommand{\mc}{\mathcal}

\newcommand{\id}{\operatorname{id}}

\newcommand{\Rm}{\operatorname{Rm}}
\newcommand{\Ric}{\operatorname{Ric}}

\newcommand{\diam}{\operatorname{diam}}

\newcommand{\vol}{\operatorname{vol}}

\newcommand{\Vol}{\operatorname{Vol}}

\newcommand{\interior}{\operatorname{int}}

\usepackage{enumitem}

\newtheoremstyle{break}
  {}
  {}
  {\itshape}
  {}
  {\bfseries}
  {.}
  {\newline}
  {}

  \theoremstyle{break}

\newtheorem{thm}{Theorem}[section]

\newtheorem{prop}{Proposition}[section]
\newtheorem{ethm}[prop]{Theorem}
\newtheorem{cor}[thm]{Corollary}

\newtheorem{lemma}[prop]{Lemma}

\theoremstyle{definition}
\newtheorem{defn}[prop]{Definition}

\theoremstyle{remark}

\makeatletter
\DeclareRobustCommand*{\mfaktor}[3][]
{
   { \mathpalette{\mfaktor@impl@}{{#1}{#2}{#3}} }
}
\newcommand*{\mfaktor@impl@}[2]{\mfaktor@impl#1#2}
\newcommand*{\mfaktor@impl}[4]{
   \settoheight{\faktor@zaehlerhoehe}{\ensuremath{#1#2{#3}}}%
   \settoheight{\faktor@nennerhoehe}{\ensuremath{#1#2{#4}}}%
      \raisebox{-0.5\faktor@zaehlerhoehe}{\ensuremath{#1#2{#3}}}%
      \mkern-4mu\diagdown\mkern-5mu%
      \raisebox{0.5\faktor@nennerhoehe}{\ensuremath{#1#2{#4}}}%
}
\makeatother

\usepackage{bm}

\begin{document}

\begin{abstract}
  The regularity of limit spaces of Riemannian manifolds with $L^p$ curvature bounds, $p > n/2$, is investigated under no apriori non-collapsing assumption. A regular subset, defined by a local volume growth condition for a limit measure, is shown to carry the structure of a Riemannian manifold. One consequence of this is a compactness theorem for Riemannian manifolds with $L^p$ curvature bounds and an a priori volume growth assumption in the pointed Cheeger--Gromov topology.

  A different notion of convergence is also studied, which replaces the exhaustion by balls in the pointed Cheeger--Gromov topology with an exhaustion by volume non-collapsed regions. Assuming in addition a lower bound on the Ricci curvature, the compactness theorem is extended to this topology. Moreover, we study how a convergent sequence of  manifolds disconnects topologically in the limit.

  In two dimensions, building on results of Shioya, the structure of limit spaces is described in detail: it is seen to be a union of an incomplete Riemannian surface and $1$-dimensional length spaces.
\end{abstract}
\title{On limit spaces of Riemannian manifolds with volume and integral curvature bounds}
\author{Lothar Schiemanowski}
\email{lothar.schiemanowski@math.uni-freiburg.de}
\address{Mathematisches Institut der Universit\"at Freiburg\\ Ernst-Zermelo-Stra{\ss}e 1\\ D--79104 Freiburg\\ Germany}
\maketitle

\section{Introduction}
Convergence and compactness of sequences of Riemannian manifolds under various geometric conditions is by now a mature subject.  A historically significant example is Cheeger's compactness theorem concerning the class of Riemannian manifolds $(M,g)$ with uniform bounds on the sectional curvature, the diameter and the volume. This result states that for any sequence of manifolds in this class, there exists a subsequence converging in the $C^{1,\alpha}$ Cheeger--Gromov topology to a new Riemannian manifold of class $C^{1,\alpha}$. This result has been improved and generalized in myriad directions. 

When a sequence of manifolds converges in the Gromov--Hausdorff topology, the limit space may fail to be a Riemannian manifold. To ensure that the limit space is a manifold of the same dimension, a crucial point is to exclude collapsing. Lower bounds on the Ricci curvature ensure that if a sequence of connected manifolds undergoes collapse, then it is collapsing everywhere. Uniform integral curvature conditions are in general not sufficient to rule out such behavior. A bound on $\int_M |\Rm_g|^p \vol_g$ for $p > \dim(M)/2$ is sufficiently strong to control the underlying Riemannian metric, but only where the manifold is not volume collapsed. One of the main goals of this article is to equip the non-collapsed part of a limit space of Riemannian manifolds with a uniform bound on this integral curvature with the structure of a Riemannian manifold. After describing these results, we will turn to another theme: convergence, when local collapsing has been ruled out.

Volume collapsing functions, introduced in the following definition, will be the central tool to distinguish collapsed and non-collapsed regions in the limit space.
\begin{defn}
  Let $(X,d,\mu)$ be a metric measure space and let $n \in \N$. The {\em ($n$-dimensional) volume collapsing function} at $x \in X$ is defined to be
  $$\nu(x) = \inf_{r \in (0,1)} \frac{\mu(B(x,r))}{\omega_n r^n},$$
  where $\omega_n$ is the volume of the $n$-dimensional unit ball in Euclidean $\R^n$.
\end{defn}
This definition will be applied to Riemannian manifolds (and their limit spaces) of a {\em fixed} dimension $n$. Since the $n$ will be determined by the context, the dependence on $n$ is suppressed from the notation. This definition bears some relation to the volume radius that has appeared in \cite{a1} (without this name) or \cite{a2}, definition 3.1. If $(X,d,\mu)$ is a Riemannian manifold $(M,g)$ with a lower Ricci curvature bound $\Ric_g \geq -(n-1) \kappa g$, then the volume comparison theorem implies that a lower bound for $\nu(x)$ can be computed in terms of $n, \kappa$ and $\Vol_g(B(x,1))$. Thus in this setting we could replace the volume collapsing function $\nu$ by the function $x \mapsto \Vol_g(B(x,1))$.
\begin{defn}
  Let $(X,d,\mu)$ be a metric measure space and let $n \in \N$. The {\em regular set} of $(X,d,\mu)$ is
  $$X^{reg} = \interior \{ x \in X : \nu(x) > 0 \}.$$
\end{defn}
\begin{wrapfigure}{r}{0.5\textwidth}
\caption{Example of $\{\nu(x) > \epsilon\}$}
\centering
\includegraphics{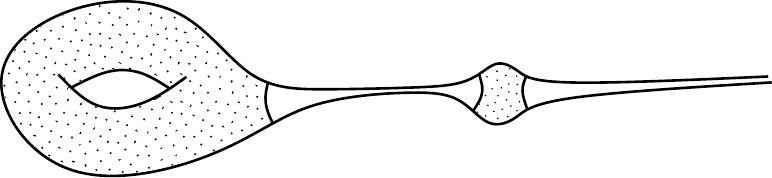}
\end{wrapfigure}

Working with partially collapsed limit spaces requires us to consider a new notion of convergence of Riemannian manifolds. In the definition below, and more generally throughout the article, the dimension of the Riemannian manifolds is always fixed, i.e.\@ $\dim M_k = \dim M = n$ for every $k \in \N$. We make one exception: the limit manifold $M$ is allowed to be empty.

\begin{defn}
  \label{DefVolExConv}
  Let $l \in \N, \alpha \in (0,1)$.
  
  A sequence $(M_k, g_k)$ of Riemannian manifolds converges in the {\em volume exhausted  $C^{l+\alpha}$ Cheeger--Gromov topology} to a Riemannian manifold $(M, g)$, if for every $\epsilon > 0$ and $R > 0$, there exists a domain $\Omega \subset M$ with
  $$\Omega \supset \{x \in M : \nu(x) > \epsilon \}$$
  and for sufficiently large $k$ there exist $C^{l+1+\alpha}$ embeddings $f_k : \Omega \to M_k$, such that
  $$f_k(\Omega) \supset \{x \in M_k : \nu(x) > \epsilon \}$$
  and $f_k^* g_k$ converges to $g$ on $\Omega$ in the $C^{l+\alpha}$ topology of definition \ref{DefConvTensors}
\end{defn}
\begin{figure}[h!]
\caption{Profiles of a sequence of rotationally invariant metrics on $S^2$. A pointed limit and the volume exhausted limit are indicated at the bottom.}
\centering
\includegraphics[width=8cm]{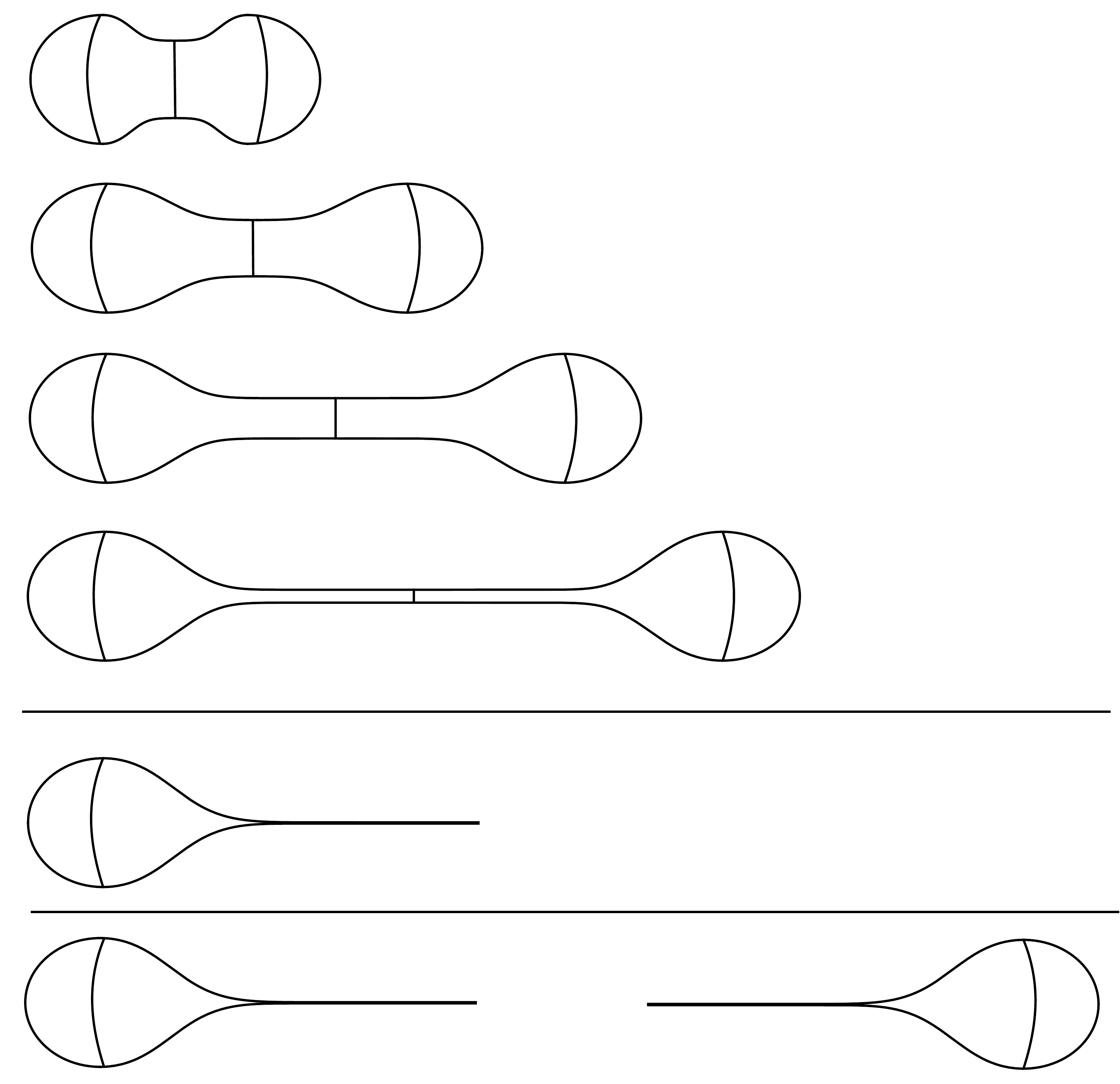}
\label{FigMetricsS2}
\end{figure}

A straightforward adaptation of pointed Cheeger--Gromov convergence gives rise to a pointed version of our volume exhausted topology.
\begin{defn}
  \label{DefPVolExConv}
  Let $l \in \N, \alpha \in (0,1)$.
  
  Let $(M_k,g_k,x_k)$ be a sequence of Riemannian manifolds. This sequence converges in the {\em pointed, volume exhausted  $C^{k+\alpha}$ Cheeger--Gromov topology} to a pointed Riemannian manifold $(M, g,x)$, if for every $\epsilon > 0$ and $R > 0$, there exists a domain $\Omega \subset M$ with
  $$\Omega \supset \{x \in M : \nu(x) > \epsilon \} \cap B(x,R)$$
  and $C^{l+1+\alpha}$ embeddings $f_k : \Omega \to M_k$, such that
  $$f_k(\Omega) \supset \{x \in M_k : \nu(x) > \epsilon \} \cap B(x_k, R)$$
  and $f_k^* g_k$ converges to $g$ on $\Omega$ in the $C^{l+\alpha}$ topology of definition \ref{DefConvTensors}
\end{defn}

With these preparations we are ready to state our first main result.
\begin{thm}
  \label{ThmRegSubconv}
  Let $\Lambda > 0, p > n/2$, $\alpha \in (0, 2 - n/p)$.

  Let $(M_k, g_k, x_k)$ be a sequence of pointed complete Riemannian manifolds satisfying
  $$\left( \int_{M_k} |\Rm_{g_k}|^p \vol_{g_k} \right)^{1/p} \leq \Lambda.$$
  Suppose $(M_k, g_k, x_k)$ is precompact in the pointed Gromov--Hausdorff topology.

  Then there is a subsequence, still denoted by $(M_k, g_k, x_k)$, with the following properties:
  \begin{enumerate}
  \item $(M_k, g_k, x_k)$ converges in the pointed, measured Gromov--Hausdorff topology to a limit space $(X,d, \mu, x)$,
  \item the regular set $X^{reg}$ has the structure of a $C^{\alpha}$ Riemannian manifold $(X^{reg}, g_{reg})$,
  \item if $x \in X^{reg}$, then $(M_k, g_k, x_k)$ converges to $(X^{reg}, g_{reg}, x)$ in the pointed, volume exhausted $C^{\alpha}$ Cheeger--Gromov topology.
  \end{enumerate}
\end{thm}

In certain situations $X^{reg} = X$. One fairly general criterium for this is recorded in the next theorem.
\begin{thm}
  \label{ThmCptVNCIntC}
  Let $\Lambda > 0, p > n/2, \alpha \in (0, 2-n/p)$ and $\varepsilon : (0, \infty) \to (0, \infty)$ a locally bounded function.
  
  Suppose $(M_k, g_k, x_k)$ is a sequence of pointed, connected complete Riemannian manifolds satisfying
  $$\left( \int_{M_k} |\Rm_{g_k}|^p \vol_{g_k} \right)^{1/p} \leq \Lambda$$
  and for all $x \in M_k$ the volume collapsing constant satisfies
  $$\nu(x) \geq \varepsilon(d_{g_k}(x,x_k)).$$
  Then there exists a subsequence converging in the pointed $C^{\alpha}$ Cheeger--Gromov topology.
\end{thm}

We note that this theorem implies the following well known theorem.
\begin{cor}
  \label{CorCptIntCRicBdd}
  Let $\Lambda > 0, v > 0, p > n/2, \alpha \in (0, 2-n/p)$.
  
  Suppose $(M_k, g_k, x_k)$ is a sequence of pointed, connected complete Riemannian manifolds satisfying
  $$\left( \int_{M_k} |\Rm_{g_k}|^p \vol_{g_k} \right)^{1/p} \leq \Lambda,$$
  $$\Ric_{g_k} \geq -(n-1),$$
  $$\Vol_{g_k} (B(x_k, 1)) \geq v.$$
  
  Then there exists a subsequence converging in the pointed $C^{\alpha}$ Cheeger--Gromov topology.
\end{cor}

In the next theorem we consider the {\em global} notion of convergence introduced in definition \ref{DefVolExConv}. It turns out one can still construct a sublimit.
\begin{thm}
  \label{ThmGlobalConvergence}
  Let $\Lambda > 0, V > 0, p > n/2, \alpha \in (0, 2-n/p)$.
  
  Suppose $(M_k, g_k)$ is a sequence of complete Riemannian manifolds satisfying
  $$\left( \int_{M_k} |\Rm_{g_k}|^p \vol_{g_k} \right)^{1/p} \leq \Lambda,$$
  $$\Vol_{g_k}(M_k) \leq V,$$
  $$\Ric_{g_k} \geq -(n-1).$$
  
  Then there exists a subsequence converging to a limit Riemannian manifold $(M,g)$ in the volume exhausted $C^{\alpha}$ Cheeger--Gromov topology.
\end{thm}
The limit may be the empty manifold. Consider for example a sequence of collapsing flat tori. Perhaps more surprisingly, the limit may also consist of infinitely many components. This may happen even if all $M_k$ are connected. Consider a cylinder of length $1$ and radius $r$. To each boundary circle we attach hyperbolic cusp. The space thus constructed has only a $C^0$ metric, but we may smooth this in such a way that the curvature remains bounded between $-1$ and $1$.  The total volume of this space goes to zero as $r$ goes to zero. We denote this space by $(C_r, g_r)$.
\begin{figure}
\caption{$(C_r,g_r)$}
\centering
\includegraphics[width=10cm]{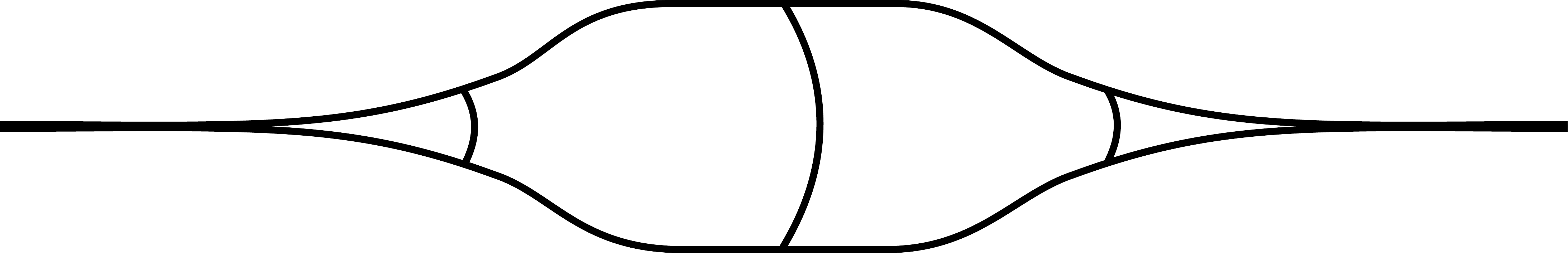}
\end{figure}
Choose $r_k$, such that $\Vol_{g_{r_k}}(C_{r_k}) \leq 2^{-k}$. Then we define $(M_k, g_k)$ by ``gluing'' the $(C_{r_1}, g_{r_1}), \ldots, (C_{r_k}, g_{r_k})$ in a chain. The ``gluing'' is done by replacing two cusps on $(M_k, g_k)$ and $(M_{k+1}, g_{k+1})$ by a hyperbolic collar. The length of the inner closed geodesic of the collar is chosen to converge to $0$ as $k$ goes to infinity.  This sequence converges in the sense of definition \ref{DefVolExConv} to the disjoint union $\sqcup_{k \in \N } (C_{r_k}, g_{r_k})$.
\begin{figure}
\caption{$(M_3,g_3)$}
\centering
\includegraphics[width=13cm]{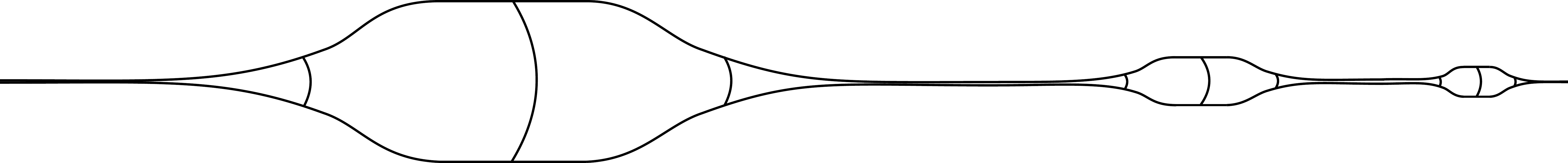}
\end{figure}

Inspired by examples like this, we study how the manifolds may $M_k$ disconnect in the volume exhausted limit. To this end we attach graphs to Riemannian manifold, which record the adjacency structure of regions where the manifold is volume collapsed and where it is not.

A precise definition of these graphs will be given in section \ref{SectCollapsingGraphs}. Here we content ourselves with a rough description of their structure.

Given a Riemannian manifold $(M,g)$ and $\epsilon > 0$, we can consider the sets
$$A_{\epsilon} = \{x \in M : \Vol_g(B(x,1)) \geq \epsilon \} \quad \text{ and} \quad B_{\epsilon} = \{x \in M : \Vol_g(B(x,1)) \leq \epsilon \}.$$
Then let $V = \pi_0(A_{\epsilon}) \cup \pi_0(B_{\epsilon})$ be the set of vertices of the graph. We place an edge between $Z_1, Z_2 \in V$ if and only if $Z_1 \neq Z_2$ and $Z_1 \cap Z_2 \neq \varnothing$. Unfortunately, the function $x \mapsto \Vol_g(B(x,1))$ may be very ill-behaved. This is why this graph is not well suited for our study. The graphs that we will consider -- denoted by $\Gamma_{\epsilon}(M,g)$ -- will arise in a similar fashion, but by substantially coarsening the sets $A_{\epsilon}$ and $B_{\epsilon}$ and modifying the condition for an edge.

In the statement of the theorem we will need the two set functions
$$v_{\min}(U) = \inf_{x \in U} \Vol_g(B(x,1)) \quad \text{and} \quad v_{\max}(U) = \sup_{x \in U} \Vol_g(B(x,1)),$$
where $(M,g)$ is a given Riemannian manifold and $U \subset M$.

\begin{thm}
  \label{ThmCollapsingGraphs}
  Let $(M_k, g_k)$ be a sequence of complete Riemannian manifolds converging in the $C^{\alpha}$ volume exhausted Cheeger--Gromov topology to a complete Riemannian manifold $(M,g)$. Suppose moreover that $M$ is a manifold with finitely many ends (cf.\@ definition \ref{DefManifoldWithEnds}) and $\Vol_g(M) < \infty$.

  Then there exists an $\epsilon_0 > 0$ with the following significance. Let $\epsilon \in (0, \epsilon_0)$.

  \begin{enumerate}
  \item The graph $\Gamma_{\epsilon}(M,g)$ is finite and depends only on the topology of $M$.
  \item Every connected component of $M$ corresponds to a  component of $\Gamma_{\epsilon}(M,g)$ and this component is a star with as many leaves as $M$ has ends.
  \item The centers of the stars in $\Gamma_{\epsilon}(M,g)$ correspond to
    $$V_{\epsilon}^{\alpha}(M,g) = \{Z \in V_{\epsilon}(M,g) : v_{\min}(Z) > 0 \}.$$
    The leaves of the stars in $\Gamma_{\epsilon}(M,g)$ correspond to
    $$V_{\epsilon}^{\omega}(M,g) = \{Z \in V_{\epsilon}(M,g) : v_{\min}(Z) = 0 \}.$$
  \item For sufficently large $k$ there exists a graph morphism $\varphi_k : V_{\epsilon}(M,g) \to V_{\epsilon}(M_k, g_k)$, which is one to one on $V_{\epsilon}^{\alpha}(M,g)$.
  \end{enumerate}
\end{thm}
\begin{figure}[ht]
  \caption{Illustration of theorem \ref{ThmCollapsingGraphs}}
\centering
\includegraphics[width=14cm]{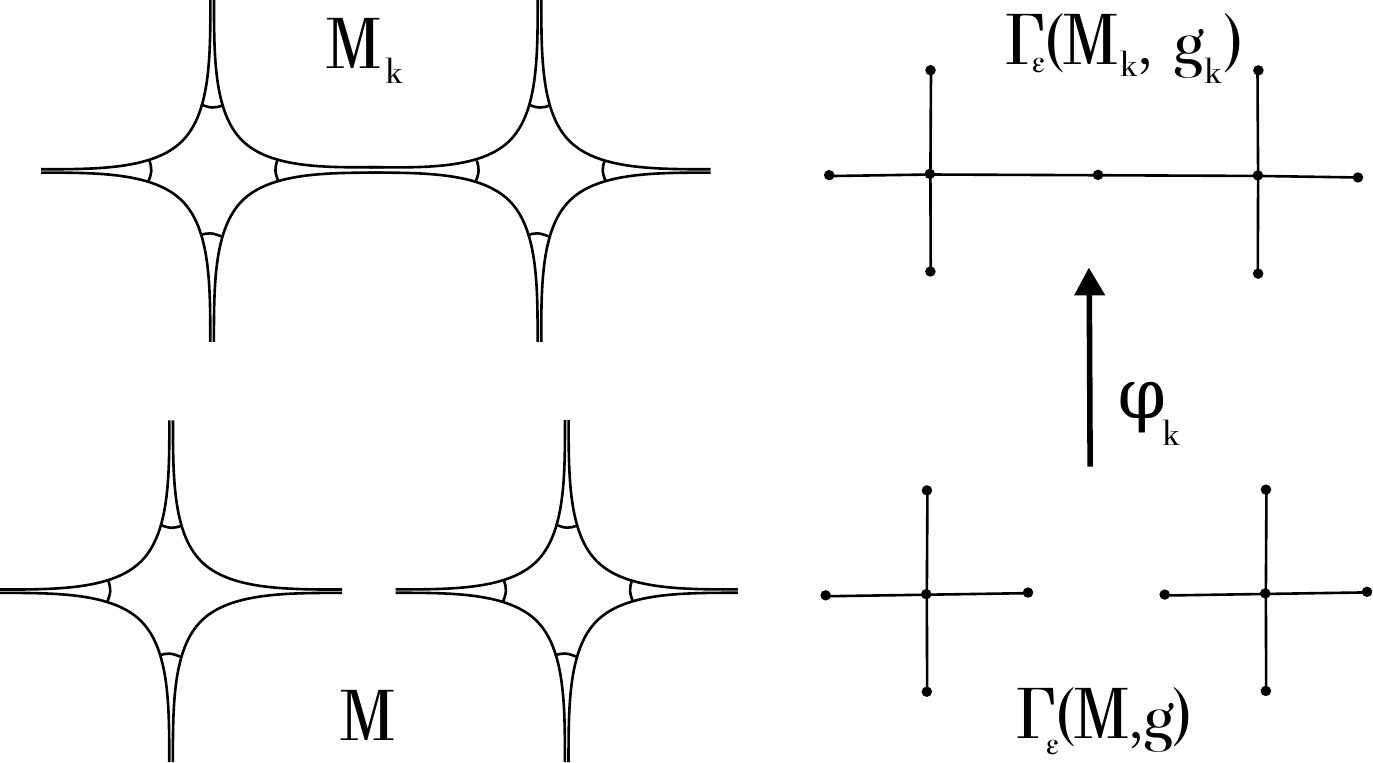}
\end{figure}
The example given before the theorem has shown that it is not guaranteed that a limit space is a manifold with finitely many ends, even if all manifolds in the sequence are connected and have uniformly bounded sectional curvature. We also note that the choice of $\epsilon$ in the proof of this theorem suppresses the geometry of $(M,g)$ completely in favor of the topology. It would be quite interesting to have a more detailled understanding of these graphs for larger $\epsilon > 0$, where the geometry becomes visible.

One might hope that the graphs $\Gamma_{\epsilon}(M_k, g_k)$ stabilize for large $k$. This is not the case, however, as we may construct a sequence similar to the one in figure \ref{FigMetricsS2} which oscillates between one and two components.

The remainder of the paper is devoted to studying surfaces. The particularly simple interaction between topology and geometry --- in the form of the Gauß--Bonnet theorem --- has ramifications for limit spaces, both in the collapsed and in the non-collapsed setting.

The first result is that theorem \ref{ThmRegSubconv} can be significantly improved by applying Shioya's results in \cite{s} on limit spaces of surfaces with bounded integral curvature.

\begin{thm}
  \label{ThmCptSurfIntC}
  Let $V, \Lambda > 0, p>1$ and $\alpha \in (0, 2 - 2/p)$.
  
  Suppose $(M_k, g_k, x_k)$ is a sequence of complete Riemannian surfaces satisfying
  $$\Vol_{g_k}(M_k) \leq V,$$
  $$\left( \int_{M_k} |K_{g_k}|^p \vol_{g_k} \right)^{1/p} \leq \Lambda.$$
  Then there exists a subsequence -- still denoted by $(M_k, g_k, x_k)$ -- with the following properties:
  \begin{enumerate}
  \item $(M_k, g_k, x_k)$ converges in the pointed, measured Gromov--Hausdorff sense to a limit space $(X,d, \mu, x)$,
  \item the set $X^{reg}$ has the structure of a $C^{\alpha}$ Riemannian surface $(X^{reg}, g)$
\item $(M_k, g_k,x_k)$ converges to $(X^{(2)}, g, x)$ in the pointed, volume exhausted $C^{\alpha}$ Cheeger--Gromov topology,
  \item the set
  $$X^{-} = \{x \in X : \mu(B(x,r)) = 0 \text{ for some } r > 0\}$$
  is an open subset of $X$, which is isometric to a disjoint union of points, line segments and circles,
\item the set $X^{\perp} = \partial X^{reg} = \partial X^-$ is a discrete subset of $X$ and
  $$X = X^{reg} \sqcup X^- \sqcup X^{\perp}.$$
\end{enumerate}
\end{thm}
\begin{wrapfigure}{r}{0.5\textwidth}
\caption{Example of a limit space: the surfaces enclosed by the solid curves are $X^{reg}$, the dotted lines are $X^-$, the thick points are $X^{\perp}$}
\centering
\includegraphics[width=8cm]{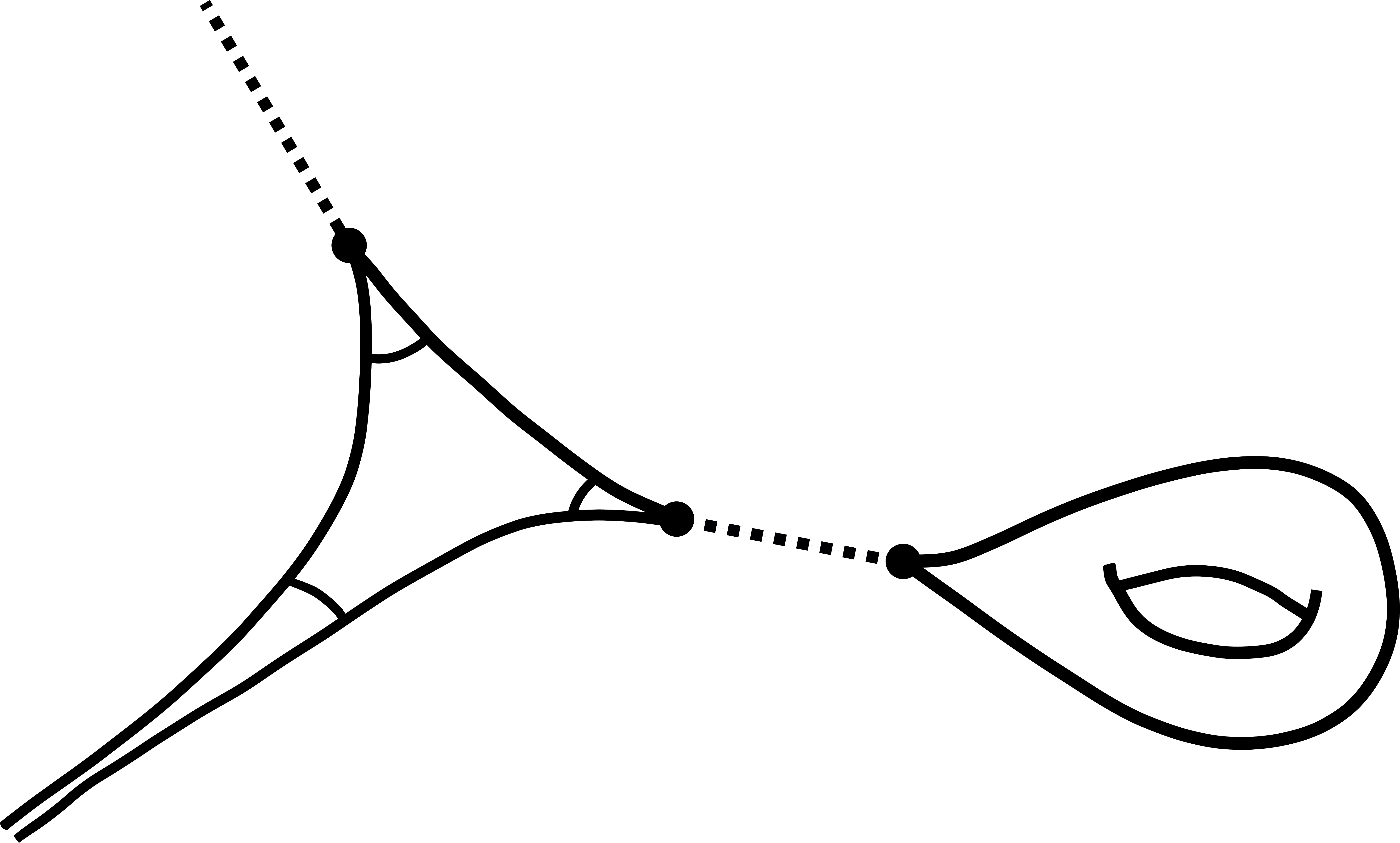}
\end{wrapfigure}

As alluded to in the beginning, these results fit into a long line of research about the limit spaces of Riemannian manifolds under various geometric conditions. In \cite{c} X.X. Chen considers Riemannian metrics on surfaces in a fixed conformal class under an $L^2$ curvature bound and finds a concentration compactness principle. The attempt to understand this from an intrinsic point of view and without assumptions on the conformal class was a strong driver of this work. For uniformly bounded curvature this has already been performed by C.\@ Bavard and P.\@ Pansu in \cite{bp}. In particular, they introduce a thick thin decomposition of surfaces in terms of the injectivity radius. Their construction inspired the definition of the graphs considered in theorem \ref{ThmCollapsingGraphs}. They also prove that after rescaling any sequence of Riemannian surfaces with bounded curvature subconverges to either another Riemannian surface or a metric graph. This inspired our construction of ``global limits'' in theorem \ref{ThmGlobalConvergence}. The difference is that we do not rescale our manifolds, so that we always get a genuine manifold of the same dimension in the limit rather than a graph. Theorem \ref{ThmCptSurfIntC} is a modest combination of our results with T. Shioya's beautiful theory of surfaces of bounded integral curvature in \cite{s}. In this context we also mention C. Debin's article \cite{d}, where a compactness theorem for surfaces of integral curvature with a non-collapsing assumption is proven. While much of this work was inspired by the two-dimensional case, most of the machinery employed and phenomena observed are actually dimension independent. The observation of the local collapsing phenomenon under integral curvature bounds is presumably folklore. The first compactness theorem for Riemannian manifolds with $L^p$ curvature bounds the author is aware of is found in D. Yang's work \cite{y1}. Indeed, theorem 2.1 in \cite{y1} is similar to theorem \ref{ThmRegSubconv}. Our theorem differs from theirs in two respects. First, the convergence is stronger in our theorem. Second, in Yang's theorem the convergence is restricted to open subsets where the condition $\nu > \epsilon$ holds for every member of the sequence. We also note that in theorem 3 of \cite{y3} and theorem 3.2 of \cite{a1}, the authors exhibit thick-thin decompositions of manifolds with $L^p$, $p > n/2$, Riemann curvature bounds. Anderson's decomposition is based on the volume radius, whereas Yang's decomposition is based on the weak injectivity radius introduced in \cite{y4}.

We close this introduction with an overview of the structure of the paper. Section \ref{SecBackground} recalls several notions we need throughout the article. In particular, two key theorems are recalled: the volume comparison theorem under integral curvature assumptions and an estimate for the harmonic radius. The proof of theorem \ref{ThmRegSubconv} is the main goal of section \ref{SecRiemMfdICB}. The proofs of theorem \ref{ThmCptVNCIntC} and \ref{CorCptIntCRicBdd} are also found in this section. Section \ref{SecVolEx} is devoted to understanding the notion of volume exhausted convergence and proving the two theorems \ref{ThmGlobalConvergence} and \ref{ThmCollapsingGraphs}. Section \ref{SecLimitSurfaces} is concerned with the structure of  limit spaces of Riemannian surfaces under an integral curvature bound, i.e.\@ this is where theorem \ref{ThmCptSurfIntC} is proven. The results of Shioya mentioned before are also recalled in detail.

\section{Background information}
\label{SecBackground}

\begin{defn}
  Let $(X_1, d_1)$, $(X_2, d_2)$ be two metric spaces and let $\epsilon > 0$.

  A map $f: X_1 \to X_2$ is called an {\em $\epsilon$-isometry}, if
  \begin{enumerate}
  \item $|d_1(x,y) - d_2(f(x), f(y)) | < \epsilon$ for every $x,y \in X_1$,
  \item $B_{\epsilon}(f(X_1)) = X_2$.
  \end{enumerate}
\end{defn}
\begin{defn}
  A sequence of metric spaces $(X_k, d_k)$ converges to a metric space $(X,d)$ in the {\em Gromov--Hausdorff topology}, if there exist $\epsilon_k$-isometries $f_k : X_k \to X$ with $\epsilon_k \to 0$.

  A sequence of pointed metric spaces $(X_k, d_k, x_k)$ converges to a metric space $(X,d, x)$ in the {\em pointed Gromov--Hausdorff topology}, if for every $R > 0$ there exist $\epsilon_k$-isometries $f_k : B(x_k, R) \to B(x, R)$ with $\epsilon_k \to 0$.
\end{defn}
\begin{defn}
  A sequence of metric measure spaces $(X_k, d_k, \mu_k)$ converges to a metric measure space  in the {\em measured Gromov--Hausdorff topology}, if there exist measurable $\epsilon_k$-isometries $f_k : X_k \to X$ with $\epsilon_k \to 0$, such that the pushforward measures $f_{k*} \mu_k$ converge weakly to the measure $\mu$.
\end{defn}

The following proposition is a straightforward application of Prokhorov's theorem and so we omit its proof.
\begin{prop}
  \label{PropMeasuredSubconv}
  Suppose $(X_k,d_k, \mu_k,x_k)$ is a sequence of pointed metric measure spaces. Suppose that $(X_k, d_k, x_k)$ converges in the pointed Gromov--Hausdorff topology.

  If there is a function $F : \R_+ \to \R_+$, such that $\mu_k(B(x_k, R)) \leq F(R)$ for every $R > 0$ and every $k \in \N$, then there exists a subsequence converging in the pointed, measured Gromov--Hausdorff topology.
\end{prop}

For definiteness, we also recall what we mean by convergence of tensor fields on manifolds.
\begin{defn}
  \label{DefConvTensors}
  Let $M$ be a manifold.
  
  A family $T_k$ of tensor fields {\em $C^{l,\alpha}$-converges} to $T$, if
  \begin{enumerate}
  \item there exists a covering of $M$ by coordinate charts $(U_s, \varphi_s)$, such that the coordinate changes are in $C^{l+1, \alpha}$,
  \item on every chart the tensor fields $T_k$ converge in $C^{l,\alpha}$ to $T$, i.e.\@
    $$\|(\varphi_s^{-1})^* (T_k - T)\|_{C^{l,\alpha}(\varphi_s(U_s))} \to 0$$
    as $k \to \infty$.
  \end{enumerate}
\end{defn}

By $\H^n$ we denote the $n$ dimensional hyperbolic space.
Let $V_n(r)$ denote the volume of the ball $B(x,r) \subset \H^n$. The function $V_n(r)$ behaves like the volume of balls in Euclidean space for sufficiently small $r$, i.e.\@ $V_n(r) \approx \omega_n r^n$ for $r \ll 1$.

The next theorem is the classical volume comparison theorem under a lower Ricci curvature bound.
\begin{ethm}
  \label{ThmClassicVolComp}
  Let $(M,g)$ be a complete $n$-dimensional Riemannian manifold and suppose \mbox{$\Ric_g \geq -(n-1)$}. Then
  $$\frac{\Vol_g(B(x,r))}{\Vol_g(B(x,R))} \geq \frac{V_n(r)}{V_n(R)}$$
  for every $x \in M$ and every $R > r > 0$.
\end{ethm}
The next theorem is a generalisation of the volume comparison theorem, which only assumes a bound on the integral  of the Ricci curvature. This theorem was proven by Petersen and Wei and appears as theorem 1.1 in \cite{pw}.
\begin{ethm}
  \label{ThmIntVolComp}
  Let $(M,g)$ be a complete $n$-dimensional Riemannian manifold. Let $x \in M$, $p > n/2$, $R > 0$. Then there exists a constant $C(n, p, R) > 0$, such that
  $$\left( \frac{\Vol_g(B(x,R))}{V_n(R)} \right)^{1/2p} - \left( \frac{\Vol_g(B(x,r))}{V_n(r)} \right)^{1/2p} \leq C k(p, R)^{1/2p}$$
  for every $r < R$, where
  $$k(p,R) = \int_{B(x,R) \cap \{\Ric_g \leq -(n-1) g\}} |\Ric_g + (n-1) g|^p \vol_g.$$
  Moreover,
  $$\Vol_g(B(x,R)) \leq (1 + C k(p,R)^{1/2p})^{2p} V_n(R).$$
\end{ethm}

\begin{defn}
  \label{DefHarmonicRadius}
  Let $(M,g)$ be a Riemannian manifold and let $p > 1$.
  
  The {\em ($W^{2,p}$)-harmonic radius} $r_H[Q,p](x)$ is defined to be the supremum of all $r > 0$, such that on the ball $B(x,r)$ there is a harmonic chart $\varphi : B(x,r) \to V \subset \R^n$ and the coefficients $g_{ij}$ of the metric $g$ in the chart satisfy
  \begin{enumerate}
  \item $Q^{-1} \delta_{ij} \leq g_{ij} \leq Q \delta_{ij}$,
  \item $\sum_{1 \leq |\beta| \leq 2} r^{|\beta| - n/p} \|\partial^{\beta} g_{ij}\|_{L^p(V)} \leq Q - 1$.
  \end{enumerate}
\end{defn}

The following theorem appears in various guises in the literature. We refer to \cite{a1}, proposition 3.1 and theorem 5.4, \cite{p1}.
\begin{ethm}
  \label{ThmHarmonicRadius}
  Let $(M,g)$ be a complete Riemannian manifold and suppose that $\epsilon, \Lambda > 0, Q>1, p > n/2$.

  There exists $\rho_0 (\epsilon, \Lambda, p, Q) > 0$, such that for any $x_0 \in M$ with
  \begin{enumerate}
  \item $\int_{B(x_0,1)} |\Rm_g|^p \vol_g < \Lambda$,
  \item $\Vol_g(B(x,r)) \geq \epsilon r^n$ for all $r \in [0,1]$ and $x \in B(x_0, 1)$,
  \end{enumerate}
  the harmonic radius at $x$ satisfies
  $$r_H[Q,p](x) > \rho_0$$
  for every $x \in B(x_0, 1/2)$.
\end{ethm}

\section{Riemannian manifolds under integral curvature bounds}
\label{SecRiemMfdICB}

\subsection{Proof of theorem \ref{ThmRegSubconv}}
For the reader's convenience and to fix notation for this section, we repeat the statement of the theorem.

Let $\Lambda > 0, p > n/2$, $\alpha \in (0, 2 - n/p)$ and let $(M_k, g_k, x_k)$ be a sequence of pointed Riemannian manifolds satisfying
$$\left( \int_{M_k} |\Rm_{g_k}|^p \vol_{g_k} \right)^{1/p} \leq \Lambda$$
and suppose the sequence is precompact in the pointed Gromov--Hausdorff topology.

Then there is a subsequence, still denoted by $(M_k, g_k, x_k)$, with the following properties:
\begin{enumerate}
\item $(M_k, g_k, x_k)$ converges in the pointed, measured Gromov--Hausdorff topology to a limit space $(X,d, \mu, x)$,
\item the regular set $X^{reg}$ has the structure of a $C^{\alpha}$ Riemannian manifold $(X^{reg}, g_{reg})$,
\item if $x \in X^{reg}$, then $(M_k, g_k, x_k)$ converges to $(X^{reg}, g_{reg}, x)$ in the pointed, volume exhausted $C^{\alpha}$ Cheeger--Gromov topology.
\end{enumerate}

The proof of this theorem follows established methods in the compactness theory of Riemannian manifolds. Indeed, we will follow the approach in \cite{p1}, Thm 2.2. For a more recent presentation along the same lines, see \cite{p2}, Thm. 11.3.6. Where the proof is essentially identical to these proofs, we will content ourselves with sketching the argument, referring to \cite{p1} or \cite{p2} for details.

In the following, we will repeatedly pass to subsequences. As usual, this will be done without changing the indexing.

The first step is to pass to a subsequence, which is convergent with respect to the pointed Gromov--Hausdorff topology. This is possible by assumption.

By theorem \ref{ThmIntVolComp}, the volume of balls in any $(M_k, g_k, x_k)$ satisfy $\Vol_{g_k} (B(x_k, R)) \leq (1 + C \Lambda^{1/2p})^{2p} V_n(R)$. Thus we may apply proposition \ref{PropMeasuredSubconv} to obtain a subsequence, which converges in the pointed measured Gromov--Hausdorff topology. This finishes the proof of (1).

The next step is to equip the set $X^{reg} = \interior \{ z \in X : \nu(z) > 0 \}$ with the structure of a Riemannian manifold.

To this end, we first need to investigate properties of the volume collapsing functions of $(M_k, g_k)$ and $(M,g)$. This is done in the following two lemmas.
\begin{lemma}
  \label{LemmaVCFLowerBound}
  Let $(M,g)$ be a complete $n$-dimensional Riemannian manifold and suppose
  $$\left( \int_M |\Rm_g|^p \vol_g \right)^{1/p} = \Lambda < \infty$$
  for some $p > n/2$.

  Suppose that $\nu(x) = \epsilon > 0$.

  Then there exist $\tilde{\epsilon} > 0$ and $r_0 > 0$, which depend only on $n, p, \Lambda$ and $\epsilon$, such that
  $$\nu|_{B(x,r_0)} \geq \tilde{\epsilon}.$$
\end{lemma}
\begin{proof}
  First, recall that $\nu(x) = \epsilon$ implies
  $$\Vol_g(B(x,r)) \geq \epsilon \omega_n r^n$$
  for every $r \in [0,1]$.

  We want to show that there exists $\tilde{\epsilon} > 0$ and $r > 0$, depending only on $n, p, \Lambda$ and $\epsilon$, such that $\nu(y) \geq \tilde{\epsilon}$ for every $y \in M$ with $d(x,y) < r_0$. Equivalently, we need to show
  $$\Vol_g(B(y,r)) \geq \tilde{\epsilon} \omega_n r^n$$
  for every such $y$.

  By the volume comparison theorem \ref{ThmIntVolComp}, we know that for any $y \in M$, the inequality
  $$\Vol_g(B(y,r)) \leq (1 + C \Lambda^{1/2p})^{2p} V_n(r)$$
  holds. Therfore, there is some $A > 0$, such that $\Vol_g(B(y,r)) \leq A r^n$ for all $r \in [0,1]$.

  Let $q$ be the midpoint between $n/2$ and $p$, i.e.\@ $q = \frac{1}{2} \left( \frac{n}{2} + p \right)$.

  Then, with the notation of theorem \ref{ThmIntVolComp}, we estimate
  \begin{align*}
    k(q,r) & = \int_{B(y,r) \cap \{\Ric_g \leq -(n-1) g\}} |\Ric_g + (n-1) g|^q \vol_g \\
           & \leq \int_{B(y,r)} |\Ric_g|^q \vol_g \\
           & \leq \left( \int_M |\Ric_g|^p \vol_g \right)^{q/p} \Vol_g(B(y,r))^{\frac{p}{p-q}}\\
           & \leq \Lambda^q A^{\frac{p}{p-q}} r^{\frac{np}{p-q}}
  \end{align*}
  Let $B = \Lambda^q A^{\frac{p}{p-q}}$. Then $k(q,r) \leq B r^{\frac{4 n p}{2p - n}}$. Note that the exponent $\frac{4np}{2p - n}$ exceeds $4n$.

  Let $0 < r < R$. Then by theorem \ref{ThmIntVolComp}
  $$\left(\frac{\Vol_g(B(y,r))}{V_n(r)}\right)^{1/2q} \geq \left(\frac{\Vol_g(B(y,R))}{V_n(R)}\right)^{1/2q} - C k(q,R)^{1/2q}.$$
  Let $\delta = d(x,y)$. If $\delta < R$, then $B(y,R) \supset B(x,R-\delta)$ and so
  $$\left(\frac{\Vol_g(B(y,r))}{V_n(r)}\right)^{1/2q} \geq \left(\frac{\Vol_g(B(x,R-\delta))}{V_n(R)}\right)^{1/2q} - C k(q,R)^{1/2q}.$$
  Assuming $R < 1$, we obtain
  $$\left(\frac{\Vol_g(B(y,r))}{V_n(r)}\right)^{1/2q} \geq \left(\frac{\epsilon \omega_n (R-\delta)^n}{V_n(R)}\right)^{1/2q} - C \left( B R^{\frac{4 n p}{2p - n}} \right)^{1/q}.$$
  If $\delta < R/2$, then
  $$\left(\frac{\Vol_g(B(y,r))}{V_n(r)}\right)^{1/2q} \geq \left(\frac{\epsilon \omega_n R^n}{2^n V_n(R)}\right)^{1/2q} - C \left( B R^{\frac{4 n p}{2p - n}} \right)^{1/q}.$$
  Now we note that on the one hand $R \mapsto R^n/ V_n(R)$ is bounded below on $[0,1]$ and that on the other hand $R \mapsto R^{\frac{4 np}{q(2 p - n)}}$ is going to zero as $R$ goes to zero. Thus there exists some $R_0 > 0$ and $C > 0$, such that
  $$\frac{\Vol_g(B(y,r))}{V_n(r)} \geq C$$
  for all $r \in [0, R_0]$ and $y \in B(x, R_0/2)$. Using that $V_n(r) \approx \omega_n r^n$ for small $r$, the claim of the lemma easily follows.
\end{proof}
\begin{lemma}
  \label{LemmaMeasBalls}
  Suppose $(X_n, d_n, \mu_n, x_k)$ converges to $(X,d, \mu, x)$ in the pointed, measured Gromov--Hausdorff topology.

  Then
  $$\limsup_{k \to \infty} \mu_k(B(x_k, r_1)) \leq \mu(B(x,r_2)) \leq \liminf_{k \to \infty} \mu_k(B(x_k, r_3))$$
  for any $0 < r_1 < r_2 < r_3$.
\end{lemma}
The proof of this lemma is a straightforward application of the definitions and so we omit it. Compare theorem 1.40 in \cite{eg}.
\begin{lemma}
  \label{LemmaVCFLimit}
  The inequality
  $$\limsup_{k \to \infty} \nu(x_k) \leq \nu(x)$$
  holds.

  There exists a function $f : \R_+ \to \R_+$ depending only on $n, \Lambda, p$ and $V$, such that if $\nu(x) > 0$, then
  $$\liminf_{k \to \infty} \nu(x_k) \geq f(\nu(x)).$$
\end{lemma}
\begin{proof}
  We start with the inequality $\limsup_{k \to \infty} \nu(x_k) \leq \nu(x)$.
  
  To this end, recall the definition
  $$\nu(x) = \inf_{r \in (0,1)} \frac{\mu(B(x,r))}{\omega_n r^n}.$$
  By lemma \ref{LemmaMeasBalls}, we have
  $$\mu(B(x,R)) \geq \limsup_{k \to \infty} \Vol_{g_k} (B(x_k, r))$$
  for any $R>r$.

  This implies that for any $\epsilon > 0$, we have
  $$\nu(x) \geq \inf_{R \in (0,1)} \limsup_{k \to \infty} \frac{\Vol_{g_k}(B(x_k, (1-\epsilon) R))}{\omega_n R^n}.$$
  The inequality then follows from the fact that $\inf \sup \geq \sup \inf$ and letting $\epsilon$ go to zero. More precisely,
  \begin{align*}
    \inf_{R \in (0,1)} \limsup_{k \to \infty} \frac{\Vol_{g_k}(B(x_k, (1-\epsilon)R))}{\omega_n R^n} & = (1- \epsilon)^n \inf_{R \in (0,1)} \inf_{k \in \N} \sup_{m \geq k} \frac{\Vol_{g_m}(B(x_m, (1-\epsilon)R))}{\omega_n ((1-\epsilon)R)^n} \\
                                                                                         & \geq (1-\epsilon)^n \inf_{k \in \N} \sup_{m \geq k} \inf_{R \in (0,1)} \frac{\Vol_{g_m}(B(x_m, (1-\epsilon)R))}{\omega_n ((1-\epsilon)R)^n} \\
                                                                                                     & = (1-\epsilon)^n \limsup_{k \to \infty} \inf_{R \in (0,1)} \frac{\Vol_{g_m}(B(x_m, (1-\epsilon)R))}{\omega_n ((1-\epsilon)R)^n} \\
    & = (1-\epsilon)^n \limsup_{k \to \infty} \inf_{r \in (0,1-\epsilon)} \frac{\Vol_{g_m}(B(x_m, r))}{\omega_n r^n}
  \end{align*}
  We have thus established
  $$\nu(x) \geq (1-\epsilon)^n \limsup_{k \to \infty} \inf_{r \in (0,1-\epsilon)} \frac{\Vol_{g_m}(B(x_m, r))}{\omega_n r^n}$$
  holds for every $\epsilon > 0$. The claim $\nu(x) \geq \limsup_{k \to \infty} \nu(x_k)$ follows.

  Now we turn to the second claim of the lemma. Assume that $\nu(x) > 0$. Then the inequality
  $$\mu(B(x,r)) \geq \nu(x) \omega_n r^n$$
  holds. By lemma \ref{LemmaMeasBalls} we have
  $$\mu(B(x,r)) \leq \liminf_{k \to \infty} \Vol_{g_k}(B(x_k, R))$$
  for every $R > r$. Thus for every $\epsilon > 0$ the inequality
  $$\liminf_{k \to \infty} \Vol_{g_k}(B(x_k, r)) \geq \nu(x) \omega_n ((1-\epsilon)r)^n$$
  holds for every $r > 0$. Since this holds for every $\epsilon > 0$, we have in fact
  $$\liminf_{k \to \infty} \Vol_{g_k}(B(x_k, r)) \geq \nu(x) \omega_n r^n$$
  for $r \in (0,1)$. It follows that for every $r \in (0,r)$ there exists an $N(r)$, such that for every $k \geq N(r)$ the inequality
  $$\Vol_{g_k}(B(x_k, r)) \geq \frac{1}{2} \nu(x) \omega_n r^n$$
  is true.

  We note that if $N(r)$ was independent of $r$, we would immediately obtain $\liminf_{k \to \infty} \nu(x_k) \geq \nu(x)/2$. However, we can not quite prove this. Instead, we will use the volume comparison theorem \ref{ThmIntVolComp} to show that there is an $N \in \N$ and a $\delta > 0$, such that for every $k \geq N$ and $r \in (0,1)$, the inequality
  $$\Vol_{g_k}(B(x_k, r)) \geq \delta \omega_n r^n$$
  holds. The constants $N$ and $\delta$ depend only on $n, p, \Lambda, V$ and $\nu(x)$. Thus
  $$\liminf_{k \to \infty} \nu(x_k) \geq \delta = \delta(n, p, \Lambda, V, \nu(x)),$$
  which implies the statement of the lemma.

  The first thing we note is that the volume comparison theorem assures us that there exists a constant depending only on $p$ and $\Lambda$, such that $\Vol_{g_k}(B(x,r)) \leq C V_n(r)$ for $r \in (0,1)$. Since $V_n(r) \lesssim r^n$ for $r \in (0,1)$, it follows that there exists some $A > 0$ such that
  $$\Vol_{g_k}(B(x,r)) \leq A r^n$$
  for $r \in (0,1)$.

  Let $q$ be the midpoint between $n/2$ and $p$, i.e.\@ $q = \frac{1}{2} \left( \frac{n}{2} + p \right)$. Exactly as in the proof of lemma \ref{LemmaVCFLowerBound} we find that there exists some $B$ depending only on $n, p, \Lambda$ such that
  $$k(q,r) \leq B r^{\frac{4 n p}{2p - n}}.$$

  By theorem \ref{ThmIntVolComp}, there exists a constant $C > 0$ depending only on $n$ and $p$, such that for any $0 < r < R < 1$, we have
  $$\left(\frac{\Vol_{g_k}(B(x_k,r))}{V_n(r)}\right)^{1/2q} \geq \left(\frac{\Vol_{g_k}(B(x_k,R))}{V_n(R)}\right)^{1/2q} - C k(q,R)^{1/2q}.$$
  On the one hand, there is a lower bound
  $$\left(\frac{\nu(x) \omega_n R^n}{2V_n(R)}\right)^{1/2q} \geq \eta > 0,$$
  which holds for every $R \in (0,1)$.

  On the other hand, the term $C k(q,R)^{1/2q}$ is going to $0$ as $R$ goes to $0$, because $k(q,R) \leq B R^{\frac{4 n p}{2p - n}}$.

  Thus we may fix some $R_0 \in (0,1)$, such that
  $$\left(\frac{\nu(x) \omega_n R_0^n}{2V_n(R_0)}\right)^{1/2q} - C k(q,R_0)^{1/2q} > \eta/2.$$
  We saw that there exists an $N$, such that for any $k > N$, we have
  $$\Vol_{g_k}(B(x_k, R_0)) \geq \frac{1}{2} \nu(x) \omega_n R_0^n.$$
  Thus we conclude
  $$\left(\frac{\Vol_{g_k}(B(x_k,r))}{V_n(r)}\right)^{1/2q} \geq \eta/2$$
  for every $r \in (0, R_0)$ and $k \geq N$. This implies that there exists some $\delta > 0$, which only depends on $\eta$ and $n$, such that
  $$\Vol_{g_k}(B(x_k, r)) \geq \delta r^n$$
  for every $r \in (0,1)$ and $k \geq N$, which is what we aimed to show.
\end{proof}

Aided with these lemmas, we can start equipping $X^{reg}$ with the structure of a Riemannian manifold. An important consequence of lemma \ref{LemmaVCFLimit} is that if $z \in X^{reg}$, i.e.\@ $\nu(z) > 0$, then for any sequence $z_k \in M_k$ converging to $z$, we have $\liminf_{k \to \infty} \nu(z_k) > 0$ also. By lemma \ref{LemmaVCFLowerBound} there exists some $r_0$ and $\tilde{\epsilon} > 0$, such that $\nu(w) > 0$ for every $w \in B(x_k, r_0)$ and every $w \in B(x_k, r_0)$ and every $k$. By theorem \ref{ThmHarmonicRadius}, this implies that there is a uniform lower bound in $k \in \N$ on the harmonic radius of $(M_k, g_k)$ at $z_k$. This uniform bound depends on $n, p, \Lambda$ and $\nu(z)$ and some arbitrary parameter $Q > 1$. We denote this uniform bound by $\rho_H(z)$. Thus for every $k$ one can fix harmonic coordinates $\varphi_k: B(z_k, \rho_H(z)) \subset M_k \to V_k \subset \R^n$ near $z_k$ satisfying the conditions of \ref{DefHarmonicRadius} and $\varphi_k(z_k) = 0$.

Depending on $\rho_H(z)$ and the choice of $Q > 1$, there exists some $r > 0$, such that $\varphi_k(B(z_k, \rho_H(z))) \supset B(0, r)$. Let $\psi_k : B(0,r) \to M_k$ be the inverse of $\varphi_k$ on $B(0,r)$. Passing to a subsequence, one can then construct a homeomorphism $\psi: B(0,r) \to U \subset X$, where $U$ is an open set containing $z$. Moreover, passing to a subsequence again one can ensure that the metric tensors $g_{ij}^k = \psi_k^* g_k$ on $B(0,r)$ converge to some limit $g_{ij}$ in $C^{\alpha}$.

For every $N \in \N$ consider the subset $\{z \in X^{reg} : \rho_H(z) \geq 1/N\}$ and let $P_N$ be a {\em $1/(2N)$-maximal} subset of this set, i.e.\@ a set maximal with respect to inclusion, such that $d(z, \tilde{z}) \geq 1/(2N)$ for every $z, \tilde{z} \in P_N$ with $z \neq \tilde{z}$. Since $X$ is separable, $P_N$ is a countable set. In fact, $B(x,R) \cap P_N$ is finite for every $R > 0$ and every $N \in \N$, because $B(x,R)$ has finite volume by theorem \ref{ThmIntVolComp} and every ball $B(z, 1/(2N))$ with $1/(2N) < \rho_H(z)$ has measure bounded below by the definition of the harmonic radius and by lemma \ref{LemmaMeasBalls}. By construction $\{B(z, \rho_H(z)) : z \in P_N\}$ is a cover of $\{z \in X^{reg} : \rho_H(z) \geq 1/N\}$. Now let $P = \bigcup_{N \in \N} P_N$. Then $P$ is also countable and $\{B(z, r(z)) : z \in P\}$ is a cover of $X^{reg}$. We can enumerate $P = \{z_l\}_{l \in \N}$.

For every $z_l \in P$ we perform the above construction of a chart $\psi^l: B(0,r_l) \to U_l \subset X^{reg}$ such that $\psi^l(0) = z_l$. Arguing by diagonalisation and using that  $P$ is countable, we may assume that the subsequence from which $\psi^l$ arises is the same for every $l \in \N$.

We have constructed a covering of $X^{reg}$ by sets homeomorphic to balls and thus $X^{reg}$ has the structure of a topological manifold.

For any $l \in \N$, let $g^l$ be the limiting metric coefficients on $B(0,r_l)$ associated to $\psi_l$. It can be shown that $\psi^l : (B(0,r_l), g^l) \to U^l \subset X$ is an isometry, i.e.\@ a distance preserving map. Hence, so are the coordinate changes $\left(\psi^{\tilde{l}} \right)^{-1} \circ \psi^l$. By results of Calabi--Hartman \cite{ch} or Taylor \cite{t}, this implies that the coordinate changes are in fact $C^{1+\alpha}$. Thus $\psi^l$ yields a $C^{1+\alpha}$ differentiable structure on $X^{reg}$ and the $g^l$ form a well-defined $C^{\alpha}$ Riemannian metric on $X^{reg}$, which we call $g_{reg}$.

It remains to show that if $x \in X^{reg}$, then $(M_k, g_k, x_k)$ converges to $(X^{reg}, g_{reg}, x)$ in the sense of definition \ref{DefPVolExConv}. Thus we have to find for every $R > 0$ and $\epsilon > 0$ an open set
$$\Omega \supset B(x,R) \cap \{z \in X : \nu(z) > \epsilon \},$$
and for every $k\geq N(R, \epsilon) > 0$ diffeomorphisms $f_k : \Omega \to \Omega_k \subset M_k$, such that
$$\Omega_k \supset B(x_k,R) \cap \{z \in M_k : \nu(z) > \epsilon \}$$
and $f_k^* g_k$ converges to $g_{reg}$.

This can be done completely analogously as in \cite{p2}, Thm. 11.3.6: for any finite collection of charts $(U_{l_1}, \psi_{l_1}), \ldots, (U_{l_r}, \psi_{l_r})$ as above, one can construct diffeomorphisms $F$ from $\Omega = \bigcup_{i=1}^r U_{l_i}$ into open sets of $M_k$, such that $F_k^*g_k$ converges to $g$.

We still need to see that for every $R > 0$ and $\epsilon >0$ the set $B(x,R) \cap \{z \in X^{reg} : \nu(z) > \epsilon \}$  is contained in finitely many chart domains $U_{l_1}, \ldots, U_{l_r}$. To see this, recall that the function $\rho_H(z)$ is bounded below in terms of $n, p, \Lambda$ and $\nu(z)$ and so there exists some $\epsilon_0 > 0$, such that $\{ z \in X^{reg} : \nu(z) > \epsilon \} \subset \{z \in X^{reg} : \rho_H(z) > \epsilon_0\}$. Hence it is sufficient to show that finitely many chart domains cover $B(x,R) \cap \{z \in X^{reg} : \rho_H(z) > \epsilon_0\}$. To this end, choose $N \in \N$, such that $1/N < \epsilon_0$. Then the charts corresponding to $P_N \cap B(x,R)$ form such a cover and this set is finite, as we remarked above. This finishes the proof of theorem \ref{ThmRegSubconv}.

\subsection{Volume collapsing excluded}
If $X^{reg} = X$, then theorem \ref{ThmRegSubconv} has the particularly useful conclusion that the whole limit space is a $C^{\alpha}$ Riemannian manifold. In this section, we consider two situations, where this is the case: theorem \ref{ThmCptVNCIntC} and corollary \ref{CorCptIntCRicBdd}.

\begin{proof}[Proof of theorem \ref{ThmCptVNCIntC}]
  Let $F : \R_+ \to \R_+$ be a locally bounded function. The class of pointed Riemannian manifolds $(M,g,x_0)$ with
  $$\nu(x) \geq \varepsilon(d(x,x_0)) \quad \text{ and } \quad \Vol_g(B(x_0, R)) \leq F(R)$$
  is precompact with respect to the pointed Gromov--Hausdorff topology. Indeed, let \mbox{$R \gg \rho > 0$}. It is well known that if for every $0 < \rho < R$ the the maximum number $N$ of disjoint $\rho$ balls in $B(x_0, R)$ is uniformly bounded in this class, then the class is precompact in the pointed Gromov--Hausdorff topology. By assumption
  $$\Vol_g(B(x,\rho)) \geq \varepsilon(d(x,x_0)) \rho^n \geq \epsilon_0 \rho^n,$$
  where $\epsilon_0$ is the minimum of $\varepsilon$ on $(0,R)$. Because $\varepsilon$ is locally bounded, this minimum is positive. On the other hand, we have
  $$\Vol_g (B(x_0,R)) \leq F(R).$$
  Let $x_1, \ldots, x_N$ be the centers of $N$ disjoint balls of radius $\rho$ contained in $B(x_0,R)$. Then
  $$N \epsilon_0 \rho^n \leq \Vol_g\left( \sqcup_{i=1}^N B(x_i, \rho) \right) \leq \Vol_g(B(x_0,R)) \leq F(R).$$
  Thus the maximum number of disjoint balls of radius $\rho$ in $B(x_0,R)$ is bounded by $\epsilon_0^{-1} \rho^{-n} F(R)$ and so the class is indeed precompact in the pointed Gromov--Hausdorff topology.
  
  By theorem \ref{ThmIntVolComp}
  $$\Vol_{g_n}(B(x_0, R)) \leq C(p,\Lambda, R).$$
  In particular, the assumptions of the argument above are satisfied for $(M_k, g_k, x_k)$  and so this sequence is indeed precompact in the pointed Gromov--Hausdorff topology. By proposition \ref{PropMeasuredSubconv} the family is also precompact with respect to the pointed measured Gromov--Hausdorff topology.

  The rest of the theorem follows from theorem \ref{ThmRegSubconv}. Indeed, let $(M_k, g_k, x_k)$ be a subsequence converging to $(X, d, \mu, x)$ in the pointed Gromov--Hausdorff topology. We claim that $X^{reg} = X$. To this end, let $y \in X$ and $y_k \in M_k$ be a sequence of points converging to $y$. Then
  $$\nu(y_n) \geq \varepsilon(d_{g_n}(y_n, x_n)).$$
  We note that since $y_n$ converges to $y$, the distances $d_{g_n}(y_n, x_n)$ converge to $d(y,x)$. In particular, there exists a uniform bound $R_0$ of $d_{g_n}(y_n, x_n)$. And so
  $$\nu(y_n) \geq \epsilon_0 = \min_{r \in (0, R_0)} \varepsilon(r) > 0.$$
  By lemma \ref{LemmaVCFLimit} it follows that $\nu(y) \geq \epsilon_0$. Since $y$ was chosen arbitrarily, we conclude that indeed \mbox{$X^{reg} = X$.}
\end{proof}

\begin{proof}[Proof of corollary \ref{CorCptIntCRicBdd}]
  Let $\Lambda > 0, v > 0, p > n/2$.
  
  Suppose $(M_k, g_k, x_k)$ is a sequence of pointed, connected Riemannian manifolds satisfying
  $$\left( \int_{M_k} |\Rm_{g_k}|^p \vol_{g_k} \right)^{1/p} \leq \Lambda,$$
  $$\Ric_{g_k} \geq -(n-1),$$
  $$\Vol_{g_k} (B(x_k, 1)) \geq v.$$
  The volume comparison theorem for a complete Riemannian metric $g$ with $\Ric_g \geq -(n-1)$ asserts that for any $x \in M_k$ and any $R > r > 0$, the volumes of the concentric metric balls $B(x,r)$ and $B(x,R)$ satisfy
  $$\frac{\Vol_{g_k}(B(x,r))}{\Vol_{g_k}(B(x,R))} \geq \frac{V_n(r)}{V_n(R)}.$$
  Note that the minimum
  $$\min_{r \in [0,1]} \frac{V_n(r)}{V_n(1) r^n} = \mu_n$$
  exists and is a positive number.
  
  Thus for any $x \in M_k$ we have
  $$\nu(x) \geq \mu_n \Vol_{g_k}(B(x,1)).$$
  On the other hand, since $M_k$ is connected, the distance $d(x, x_k)$ is finite for any $x_k \in M_k$ and so the volume comparison theorem also gives 
  $$\Vol_{g_k} (B(x,1)) \geq  \frac{V_n(d(x,x_0)+1)}{V_n(1)} \Vol_{g_k}(B(x,d(x,x_0)+1)) \geq \frac{V_n(d(x,x_0)+1)}{V_n(1)}v .$$
  The sequence $(M_k, g_k, x_k)$ thus satisfies the conditions of theorem \ref{ThmCptVNCIntC} with $\varepsilon(r) = \frac{v}{V_n(1)} V_n(r+1)$. The conclusions of this corollary coincide with the conclusions of theorem \ref{ThmCptVNCIntC} and so the proof of the corollary is finished.
\end{proof}

\section{Volume exhausted convergence}
\label{SecVolEx}

\subsection{Constructing a global sublimit}
In this section we prove theorem \ref{ThmGlobalConvergence}.

Suppose $V, \Lambda > 0, p > 1$.

Assume that $(M_k, g_k)$ is a sequence of complete Riemannian manifolds with
\begin{enumerate}
\item $\Vol_{g_k}(M_k) \leq V$,
\item $\left( \int_{M_k} |\Rm_{g_k}|^p \vol_{g_k} \right)^{1/p} \leq \Lambda$,
\item $\Ric_{g_k} \geq -(n-1)$.
\end{enumerate}
To prove the theorem we need to find a subsequence, which converges in the volume exhausted $C^{\alpha}$ Cheeger--Gromov topology to a limiting manifold $(M,g)$.

To this end consider the set $\mc{S}$ of sequences $(x_k)_{k\in \N}$, where $x_k \in M_k$. The reader should keep in mind that this definition depends on the sequence of manifolds $M_k$. In the following we will pass to subsequences without changing the notation for this set. We say that two sequences $(x_k)_k, (y_k)_k \in \mc{S}$ are equivalent, if
$$\sup_{k \in \N} d_{g_k}(x_k, y_k) < \infty.$$
A sequence $(x_k)_k \in \mc{S}$ is called {\em admissible}, if
$$\inf_{k \in \N} \Vol_{g_k} (B(x_k, 1)) > 0.$$
The set of admissible sequences will be denoted by $\mc{S}_+$. As a result of the volume comparison theorem, the assumption $\Ric_{g_k} \geq -(n-1)$ implies that the volumes of balls of radius $1$ are comparable if their distance is finite. Thus $\mc{S}_+$ is closed in $\mc{S}$ with respect to the equivalence relation.

Note that $\mc{S}_+$ may well be empty. Consider for instance a collapsing sequence of flat tori. If $\mc{S}_+$ is empty, this implies that for every $\epsilon > 0$, there exists an $N \in \N$, such that $\{x \in M_k : \Vol_{g_k}(B(x,1)) > \epsilon \}$ is empty for all $k \geq N$. In this case, $(M_k, g_k)$ converges to the empty manifold in the volume exhausted Cheeger--Gromov topology.

Now consider the quotient $\faktor{\mc{S}_+}{\sim}$.
\begin{prop}
  There is a subsequence of $(M_k, g_k)$, such that the set $\faktor{\mc{S}_+}{\sim}$ is finite or countable.
\end{prop}
\begin{proof}
  Assume that the proposition is false. Then for every subsequence $\faktor{\mc{S}_+}{\sim}$ is uncountable. 

  Pick for every element $E$ in $\faktor{\mc{S}_+}{\sim}$ a representative sequence $(x^E_k)_k \in E$. This way we may define a function
  $$f: \faktor{\mc{S}_+}{\sim} \to \R_+$$
  $$E \mapsto \inf_{k \in \N} \Vol_{g_k} (B(x_k^E, 1)).$$
  By the pigeonhole principle, there exists some $n \in \N$, such that $S = f^{-1}((1/n, \infty))$ is an infinite set.

  By definition of the equivalence relation, we have that
  $$\sup_{k \in \N} d_{g_k} ( x_k^{E_1}, x_k^{E_2}) = \infty$$
  for any distinct $E_1, E_2 \in S$. Passing to a subsequence we may assume that in fact
  $$\lim_{k \to \infty} d_{g_k} ( x_k^{E_1}, x_k^{E_2}) = \infty.$$
  Choose $N = 2n \lceil V \rceil$ equivalence classes $E_1, \ldots, E_{N}$ in $S$. Inductively passing to a subsequences, we may assume that
  $$\lim_{k \to \infty} d_{g_k} ( x_k^{E_i}, x_k^{E_j}) = \infty$$
  for all $1 \leq i < j \leq N$. By assumption $\inf_{k \in \N} \Vol_{g_k} (B(x_k^E, 1)) > 1/n$ for every $E \in S$. For sufficiently large $k$, we thus get
  $$\Vol_{g_k}(M_k) \geq \sum_{i=1}^{2n} \Vol_{g_k}(B(x_k^{E_i}, 1)) \geq N/n = 2 \lceil V \rceil > V,$$
  since the balls are pairwise disjoint. Clearly, this contradicts the standing assumption that \mbox{$\Vol_{g_k}(M_k) \leq V$}.
\end{proof}
Hence we may assume that $(M_k, g_k)$ is such that $\faktor{\mc{S}_+}{\sim}$ is countable and thus we identify $\faktor{\mc{S}_+}{\sim}$ with $\N$ and choose for every $l \in \N = \faktor{\mc{S}_+}{\sim}$ a sequence $x_k^l$ with the property $\inf_{k \in \N} \Vol_{g_k}(B(x_k,1)) > 0$. (If $\faktor{\mc{S}_+}{\sim}$ is finite, we instead identify $\faktor{\mc{S}_+}{\sim}$ with $\{1, \ldots ,  N \}$. This does not change the argument below.)

Recapping, we have selected a subsequence of $(M_k, g_k)$ and from this subsequence we have defined $x_k^l \in M_k$, such that
\begin{enumerate}
\item $\liminf_{k \in \N} \Vol_{g_k}(B(x_k^l, 1)) =: \epsilon_l > 0$,
\item $\lim_{k \in \N} d_{g_k}(x_k^{l_1}, x_k^{l_2}) = \infty$, if $l_1 \neq l_2$.
\end{enumerate}
By theorem \ref{CorCptIntCRicBdd}, we can find for every $l$ a subsequence of $(M_k, g_k, x_k^l)$, which converges in the pointed $C^{\alpha}$ Cheeger--Gromov topology. This can be done iteratively, so that for every $l \in \N$, the subsequence chosen for $(M_k, g_k, x_k^{l+1})$ is a subsequence of the subsequence chosen for $(M_k, g_k, x_k^l)$. We denote these subsequences by $\xi(k, l)$, i.e.\@ $(M_{\xi(k,l)}, g_{\xi(k,l)}, x_{\xi(k,l)}^l)$ converges in the pointed $C^{\alpha}$ Cheeger--Gromov topology. Choosing $\sigma(k) = \xi(k,k)$, for every $l \in \N$, the sequence $(M_{\sigma(k)}, g_{\sigma(k)}, x_{\sigma(k)}^l)$ converges in the pointed $C^{\alpha}$ Cheeger--Gromov topology.

Going forward, to lighten the notational load, the sequence $(M_{\sigma(k)}, g_{\sigma(k)}, x_{\sigma(k)}^l)$ will again be denoted by $(M_k, g_k, x_k^l)$. For $l \in \N$ we denote the limit space by $(M^l, g^l, x^l)$.

Now we claim that $(M_k, g_k)$ converges in the volume exhausted $C^{\alpha}$ Cheeger--Gromov topology to $\sqcup_{l \in \N} M^l$ equipped with the metric $g$, which coincides with $g^l$ on $M^l$.

To this end let $\epsilon > 0$. First, we note that there exists an $R > 0$, such that
$$\bigcup_{l \in \N} B(x_k^l, R) \supset \{x \in M_k : \Vol_{g_k} (B(x, 1)) > \epsilon \}$$
for all sufficiently large $k$. Indeed, if not, there would exist a sequence $z_k \in M_k$ with $\Vol_{g_k}(B(z_k,1)) > \epsilon$ for every $k$ and $d_{g_k}(z_k, x_k^l) \to \infty$ for every $l \in \N$, in contradiction to the definition of the $x_k^l$.

For every $l \in \N$ we have $B(x^l, R) \supset \{x \in M^l : \Vol_{g^l} (B(x, 1)) > \epsilon \}$ by convergence. Taking the union, we obtain
$$\bigcup_{l \in \N} B(x^l, R) \supset \left\{x \in M : \Vol_g (B(x, 1)) > \epsilon \right\}$$
By the definition of the pointed $C^{\alpha}$ Cheeger--Gromov topology, we know that for every $l\in \N$ and every $R$ there exists an open set $\Omega_l \supset B(x^l, R)$ and embeddings $f^l_k : \Omega_l \to M_k$, such that $f_k^l(\Omega_l) \supset B(x_k^l,R)$, such that $f_k^{l*} g_k$ converges to $g$ in $C^{\alpha}$ on compact subsets of $\Omega_l$. Note that the subsets $\Omega_l$ are pairwise disjoint, if they are -- as we may assume -- contained in $B(x^l, 2R)$. In that case, the sets $f_k^l (\Omega_l)$ are also pairwise disjoint, if $k$ is sufficiently large.

Hence we may define $\Omega = \cup_{l \in \N} \Omega_l$ and embeddings $f_k : \Omega \to M_k$, which then evidently satisfy the conditions of definition \ref{DefVolExConv}. This finishes the proof of theorem \ref{ThmGlobalConvergence}.

\subsection{Collapsing graphs}
\label{SectCollapsingGraphs}
In this section we prove theorem \ref{ThmCollapsingGraphs}. The proof is broken into a series of lemmas.

\begin{defn}
  \label{DefManifoldWithEnds}
  A smooth $n$ dimensional manifold $M$ is called a {\em manifold with finitely many ends} if there exists a compact $n$ dimensional submanifold $K$ with boundary and $\Omega = M \backslash K$ is diffeomorphic to $\partial K \times (0, \infty)$.
\end{defn}
One of the assumptions of the theorem was that $M$ is a manifold with finitely many ends. We identify $\Omega$ with $X \times (0, \infty)$, where $X = \partial K$.

The ``finite'' in the definition is justified because we assumed $K$ to be compact. Thus $\partial K = X$ is compact as well and in particular has only finitely many components.

Another assumption of the theorem is that $(M,g)$ is a complete Riemannian manifold and that $\Vol_g(M) = V < \infty$. 

\begin{lemma}
  For any $p \in M$ and $x \in X$ we have
  $$\lim_{t \to \infty} d_g(p, (x,t)) = \infty,$$
  $$\lim_{t \to \infty} \Vol_g(B((x,t),1)) = 0.$$
\end{lemma}
This lemma follows easily from the assumption of completeness and finiteness of the volume of $(M,g)$ and so we omit its proof.

From the lemma also follows
$$\sup_{(x,t) \in X \times (T, \infty)} \Vol_g(B((x,t),1)) \to 0$$
as $T \to \infty$.

Define
$$v : M \to \R_+,$$
$$x \mapsto \Vol_g(B(x,1)).$$
The next lemma is a consequence of the previous one.
\begin{lemma}
  \label{LemmaVolEnds}
  For any $\epsilon > 0$, there exists a $T > 0$, such that
  $$X \times (T, \infty) \subset \{v \leq \epsilon\}.$$
  For any $T > 0$, there exists a $\delta > 0$, such that
  $$\{v \leq \delta \} \subset X \times (T, \infty).$$
\end{lemma}

Given a Riemannian manifold $(M,g)$, $\epsilon > 0$ and $\lambda_+ > 1 > \lambda_0 > \lambda_- > 0$, we will construct a graph as follows.

Let
$$B = \{v \leq \lambda_0 \epsilon\}.$$
Next we define
$$C = \bigcup \left\{ Z \in \pi_0(B): v_{\min}(Z) \leq \lambda_- \epsilon \right\}.$$
We then let
$$D = C \cup \bigcup \left\{Z \in \pi_0(\overline{M \backslash C}) : v_{\max}(Z) \leq \epsilon \right\}.$$
Now we define
$$V_{\epsilon}(M,g) = \pi_0(D) \cup \{ Z \in \pi_0 (\overline{M \backslash D}) : v_{\max}(Z) > \lambda_+ \epsilon\}.$$

The set $V_{\epsilon}(M,g)$ will be the set of vertices of our graph $\Gamma_{\epsilon}(M,g)$.

For $U, V \subset M$ we define
$$\delta_g(U,V) = \inf_{(u,v) \in U\times V} d_g(u,v).$$
To define the edge structure, we say that there is an edge between $Z_1, Z_2 \in V_{\epsilon}(M,g)$ if for any $Z_3 \in V_{\epsilon}(M,g)$ with $Z_3 \neq Z_1, Z_2$ we have
$$\delta_g(Z_1, Z_2) < \delta_g(Z_1, Z_3) + \delta_g(Z_3, Z_2).$$

A brief comment on the construction: as indicated in the introduction, we would like to define the vertices to be the connected components of $\{v \leq \epsilon\}$ and of $\{v \geq \epsilon\}$ and define edges if and only if two such components intersect non-trivially. But since the function $v$ may be very unruly, one may get a truly unmanagable graph out of this. The iterative coarsenings allow us to get a handle on the graph. Two kinds of components are left in the final graph. One type with $v_{\min}(Z) \leq \lambda_- \epsilon$ and one type with $v_{\max}(Z) \geq \lambda_+ \epsilon$. In addition, when passing from $C$ to $D$, we merge components with $v_{\min}(Z) \geq \lambda_- \epsilon$ and $v_{\max}(Z) \leq \lambda_0 \epsilon$ with those satisfying $v_{\min}(Z) \leq \lambda_- \epsilon$. The purpose of this last step is to ensure that all of $M$ is covered by the components in $V_{\epsilon}(M,g)$.

The choice of $\lambda_-, \lambda_0, \lambda_+$ will be made in dependence of $(M,g)$ and is contingent on how the (somewhat aribtrary) end structure and the volume function $v$ interact.

\begin{lemma}
  \label{LemmaComponentsVolEnds}
  There exists $\epsilon_0 > 0$ such that for any $\epsilon \in (0, \epsilon_0)$, any $\lambda_0 \in (0,1)$ and $\lambda_- < \lambda_0$ sufficiently small there are bijections
  $$\pi_0(D) \to \pi_0(X \times (0, \infty))$$
  and
  $$\pi_0(\overline{M \backslash D}) \to \pi_0(K).$$
\end{lemma}
\begin{proof}
  For any $T > 0$ we may define
  $$H_T : \pi_0(X \times (0, \infty)) \to \pi_0(X \times (T, \infty))$$
  $$W \mapsto W \cap X \times (T, \infty).$$
  This is obviously a bijection. Any $Z \in \pi_0(X \times (T, \infty))$ is of the form $L \times (T, \infty)$ and so the inverse of $H_T$ is given by $H_T^{-1}(Z) = H_T^{-1}(L \times (T, \infty)) = L \times (0, \infty)$.

  Using lemma \ref{LemmaVolEnds} we choose first $\epsilon_0 > 0$, such that $\{v \leq \epsilon_0 \} \subset X \times (0, \infty)$. Then let $\epsilon \in (0, \epsilon_0)$ and $\lambda_0 \in (0,1)$. Applying lemma \ref{LemmaVolEnds} again, we find a $T > 0$, such that $X \times (T, \infty) \subset \{v \leq \lambda_0 \epsilon\}$. Now we choose $\lambda_- \in (0, \lambda_0)$, such that $\{v \leq \lambda_- \epsilon\} \subset X \times (T, \infty)$. In conclusion, we have
  $$\{v \leq \lambda_- \epsilon\} \subset X \times (T, \infty) \subset \{v \leq \lambda_0 \epsilon\} \subset \{v \leq \epsilon\} \subset X \times (0, \infty).$$
  With this choice of $\lambda_, \lambda_0, \lambda_+$ it follows that
  $$X \times (T, \infty) \subset D \subset X \times (0, \infty).$$
  Thus we can define maps
  $$F : \pi_0(X \times (T, \infty)) \to \pi_0(D) \quad \text{and} \quad G: \pi_0(D) \to \pi_0(X \times (0, \infty))$$
  by letting $F(Z)$ be the unique component in $\pi_0(D)$, which contains $Z$, and likewise we define $G(Z)$ to be the unique component in $\pi_0(X \times (0, \infty))$, which contains $Z$.

  We claim that $G \circ F = H_T^{-1}$ and $F \circ H_T \circ G = \id$.

  Let $Z = L \times (T, \infty) \in \pi_0(X \times (T, \infty))$. Then $G(F(Z)) \supset F(Z) \supset Z = L \times (T, \infty)$. The connected set containing $L \times (T, \infty)$ in $X \times (0, \infty)$ is $L \times (0, \infty)$. Thus $G(F(Z)) = L \times (0, \infty)$ and we recall that $H_T^{-1}( L \times (T, \infty)) = L \times (0, \infty)$.

  Conversely, suppose $Z \in \pi_0(D)$. Then $G(Z) = L \times (0, \infty) \supset Z$. Since $Z \supset L \times (T, \infty)$, it follows that $H_T(G(Z)) = L \times (T, \infty) = Z \cap X \times (T, \infty)$. Then by definition $F(H_T(G(Z)))$ is the component of $D$ containing $Z \cap X \times (T, \infty)$. But this is clearly $Z$ and so $F(H_T(G(Z)) = Z$.

  We now address the second part. Here we first note that we can define a map
  $$L : \pi_0(K) \to \pi_0(M)$$
  by taking any component in $\pi_0(K)$ and mapping it to the component containing $\pi_0(M)$.

  The inclusions
  $$K \subset \{v \geq \epsilon\} \subset \overline{M \backslash D} \subset M$$
  suggest to define maps
  $$I : \pi_0(\overline{M \backslash D}) \to \pi_0(M) \quad \text{and} \quad J: \pi_0(K) \to \pi_0(\overline{M \backslash D}),$$
  by the now familiar scheme.

  Here we claim $J \circ L^{-1} \circ I = \id$ and $I \circ J = L$. The argument is completely analogous to the one before.
\end{proof}
\begin{lemma}
  \label{LemmaVMax}
  We can choose $\epsilon >0$, $\lambda_-, \lambda_0, \lambda_+$, such that for every $Z \in \pi_0(\overline{M \backslash D})$,
  $$v_{\max}(Z) \geq \lambda_+ \epsilon.$$
\end{lemma}
\begin{proof}
  First let $\epsilon_0$ as in lemma \ref{LemmaComponentsVolEnds}.

  In particular, the choice there implies $K \subset \overline{M \backslash D}$. Now let
  $$\mu_0 = \min_{Z \in \pi_0(K)} v_{\max}(Z).$$
  This is well defined because each $Z \in \pi_0(K)$ is compact and $\pi_0(K)$ is a finite set. 
  
  For any $\epsilon < \mu_0$, any $\lambda_0$, $\lambda_-$ according to lemma \ref{LemmaComponentsVolEnds}, and any $\lambda_+$ such that $\lambda_+ \epsilon \geq \mu_0$, it follows that every component $Z$ of $\overline{M \backslash D}$ contains a component of $K$ and consequently $v_{\max}(Z) \geq \mu_0 \geq \lambda_+ \epsilon$.
\end{proof}

\begin{lemma}
  \label{LemmaGraphStructure}
  One can choose $\epsilon >0$, $\lambda_-, \lambda_0, \lambda_+$, such that
  \begin{itemize}
  \item the graph $\Gamma_{\epsilon}(M,g)$ is finite and depends only on the topology of $M$,
  \item every connected component of $M$ corresponds to a component of $\Gamma_{\epsilon}(M,g)$ and this component is a star with as many leaves as $M$ has ends,
  \item the centers of the stars in $\Gamma_{\epsilon}(M,g)$ correspond to
    $$V_{\epsilon}^{\alpha}(M,g) = \{Z \in V_{\epsilon}(M,g) : v_{\min}(Z) > 0 \}$$
    and the leaves of the stars in $\Gamma_{\epsilon}(M,g)$ correspond to
    $$V_{\epsilon}^{\omega}(M,g) = \{Z \in V_{\epsilon}(M,g) : v_{\min}(Z) = 0 \}.$$
  \end{itemize}
  
\end{lemma}
\begin{proof}
  Choose $\epsilon > 0$, $\lambda_-, \lambda_0, \lambda_+$ as indicated by lemma \ref{LemmaVMax}.

  Recall that
  $$V_{\epsilon}(M,g) = \pi_0 (D) \cup \{ Z \in \pi_0(\overline{M \backslash D}) : v_{\max}(Z) > \lambda_+ \epsilon\}.$$
  According to lemma \ref{LemmaVMax} $ \{ Z \in \pi_0(\overline{M \backslash D}) : v_{\max}(Z) > \lambda_+ \epsilon\} = \pi_0(\overline{M\backslash D})$. According to lemma \ref{LemmaComponentsVolEnds} $\pi_0(D)$ is in one to one correspondence with $\pi_0(X \times (0, \infty))$ and $\pi_0(\overline{M \backslash D})$ is in one to one correspondence with $\pi_0(K)$. Moreover, if $Z \in \pi_0(D)$ we have $v_{\min}(Z) = 0$ and if $Z \in \pi_0(\overline{M \backslash D})$ we have $v_{\min}(Z) \geq \lambda_- \epsilon > 0$.

  With these observations the proof of the lemma now reduces to the following claim: there is an edge between two vertices $Z_1, Z_2 \in V_{\epsilon}(M,g)$ if and only if $Z_1$ and $Z_2$ are in the same component of $M$ and one component is in $\pi_0(\overline{M \backslash D})$ and the other is in $\pi_0(D)$.

  First note that if $Z_1$ and $Z_2$ are not in the same component then there certainly no edge between $Z_1$ and $Z_2$ as $\delta_g(Z_1, Z_2) = \infty$. So we now assume $Z_1, Z_2 \subset N$, where $N$ is one component of $M$. There are two situations of $Z_1$ and $Z_2$ two consider:
  \begin{itemize}
  \item $Z_1 \in \pi_0(\overline{M \backslash D}), Z_2 \in \pi_0(D)$ (or vice versa),
  \item $Z_1, Z_2 \in \pi_0(D)$.
  \end{itemize}
  There is precisely one element $\pi_0(\overline{M \backslash D})$, which is contained in $N$, and it intersects every element of $\pi_0(D)$, which is contained in $N$. Thus in the first situation we have $\delta_g(Z_1, Z_2) = 0$. In the second situation we note that $Z_1$ and $Z_2$ are disjoint, closed subsets of $X \times (0, \infty)$ and thus $\delta_g(Z_1, Z_2) > 0$. Together these observations imply that there is an edge between $Z_1 \in \pi_0 (\overline{M \backslash D})$ and $Z_2 \in \pi_0(D)$ and no edge between $Z_1, Z_2 \in \pi_0(D)$.

  This shows the second claim of the lemma and the last follows from an observation we already made above, namely the identities
  $$\{Z \in V_{\epsilon}(M,g) : v_{\min}(Z) = 0 \} = \pi_0(D), \{Z \in V_{\epsilon}(M,g) : v_{\min}(Z) > 0 \} = \pi_0(\overline{M \backslash D}).$$
\end{proof}

\begin{lemma}
  Suppose $(M_k, g_k)$ satisfies $\Vol_{g_k}(M_k) \leq V$. Suppose $(M_k, g_k)$ converges to $(M,g)$ in the $C^{\alpha}$ volume exhausted Cheeger--Gromov topology.

  For sufficiently small $\epsilon > 0$, there exist $\lambda_-, \lambda_0, \lambda_+$ depending on $(M,g)$, such that
  \begin{itemize}
  \item the conclusions of lemma \ref{LemmaGraphStructure} remain true,
  \item for sufficiently large $k \in \N$, there exists a graph morphism
    $$\varphi_k : V_{\epsilon}(M,g) \to V_{\epsilon}(M_k,g_k),$$
    i.e.\@ if there is an edge between $Z_1, Z_2 \in V_{\epsilon}(M,g)$, then there is an edge between $\varphi_k(Z_1)$ and $\varphi_k(Z_2)$,
  \item $\varphi_k$ is surjective,
  \item $\varphi_k$ is injective on $V_{\epsilon}^{\alpha}(M,g)$.
  \end{itemize}
\end{lemma}
\begin{proof}
  Choose $\epsilon, \lambda_+, \lambda_0, \lambda_-$ and $T_5 > T_4 > T_3 > T_2 > T_1 > 0$, such that
  \begin{equation}
    X \times (T_4, \infty) \subset \{v \leq \lambda_- \epsilon/2 \} \subset \{v \leq 2 \lambda_- \epsilon \} \subset X \times (T_3, \infty), \label{Inclusion1}
  \end{equation}
  \begin{equation}
    X \times (T_3, \infty) \subset \{v \leq \lambda_0 \epsilon/2 \} \subset \{v \leq 2 \lambda_0 \epsilon \} \subset X \times (T_2, \infty), \label{Inclusion2}
  \end{equation}
  \begin{equation}
    X \times (T_2, \infty) \subset \{v \leq \lambda_+ \epsilon/2 \} \subset \{v \leq 2 \lambda_+ \epsilon \} \subset X \times (T_1, \infty),\label{Inclusion3}
  \end{equation}
  lemma \ref{LemmaGraphStructure} is satisfied and for any component $L$ of $X$. We also let $\lambda_1 > 0$ and $T_4$ be such that
  \begin{equation}
    X \times (T_5, \infty) \subset \{v \leq \lambda_1 \epsilon/2 \} \subset \{v \leq 2 \lambda_1 \epsilon \} \subset X \times (T_4, \infty)\label{Inclusion3}
  \end{equation}
  and
  \begin{equation}
    \label{ineq:distanceassumption}
    \delta_g(L_i \times \{T_3\}, L_i \times \{T_5\}) > 4 \delta_g(L_j \times \{T_1\}, L_j \times \{T_3\}) 
  \end{equation}
  for any components $L_i, L_j$ of $X$.

  Throughout this proof, $D_k$ denotes the subset of $M_k$ obtained in the same way $D$ is obtained from $M$. We also denote $\kappa = \min_{(x,t) \in X \times (0, T_5)} v((x,t))$.
  
  By the definition of volume exhausted Cheeger--Gromov convergence, there exist for any $\eta, \theta > 0$
  \begin{itemize}
  \item $N \in \N$,
  \item an open set $\Omega \supset \{x \in M : \nu(x) > \eta \}$,
  \item for every $k \geq N$ a diffeomorphism $f_k : \Omega \to \Omega_k \subset M_k$, such that $\Omega_k$ is open and $\Omega_k \supset \{x \in M_k : \nu(x) > \eta \}$
  \end{itemize}
  and the inequality $\|f_k^* g_k - g \|_{C^{\alpha}(\Omega)} < \theta$ holds for every $k \geq N$.

  Since $\nu(x) \leq \omega_n^{-1} v(x)$, it follows that $\{\nu > \eta \} \supset \{v > \omega_n \eta\}$. We now choose $\eta = \omega_n^{-1}\kappa/8$, so that
  $$\{\nu > \eta \} \supset \{v > \kappa / 8\}.$$
  Let $v_k : \Omega_k \to \R_+$ be defined by $x \mapsto \Vol_{g_k}(B(x,1))$. By assumption on $f_k$, if $B(x,1) \subset \Omega$ and $B(f_k(x),1) \subset \Omega_k$, then
  \begin{equation}
    \label{IneqVolSeqVsLim}
    |v(x) - v_k(f_k(x))| < 2^{-100} \kappa,
  \end{equation}
  if we choose $\theta$ sufficiently small. Moreover, we can choose $\theta$, such that
  $$f_k : (\Omega, d_g) \to (\Omega_k , d_{g_k})$$
  is an $2^{-100}$-almost isometry, i.e.\@
  \begin{equation}
    \label{ineq:almostisometry}
    |d_{g_k}(f_k(x_1), f_k(x_2)) - d_g(x_1, x_2)| \leq 2^{-100}
  \end{equation}
  for any $x_1, x_2 \in \Omega$.
  
  We note that the definition of $\eta$ implies
  $$K \cup X \times (0, T_5) \subset \Omega.$$

  {\bf Definition of $\bm{\varphi_k}$:}

  First, suppose $Z \in V_{\epsilon}^{\alpha}(M,g)$. Let $N \subset M$ be the component of $M$ containing $Z$. Then
  $$Z \supset \{v \geq \lambda_+ \epsilon\} \cap N.$$
  This implies
  $$Z \supset \left(K \cup X \times (0, T_1) \right) \cap N$$
  by inclusion (\ref{Inclusion3}). Now $\left(K \cup X \times (0, T_1) \right) \cap N$ is connected and thus $f_k(\left(K \cup X \times (0, T_1) \right) \cap N)$ is connected. Moreover, because of inequality (\ref{IneqVolSeqVsLim}) and again inclusion (\ref{Inclusion3}), it follows that
  $$f_k (K \cup X \times (0, T_1)) \subset \{v \geq (1 - 2^{-100}) 2 \lambda_+ \epsilon\} \subset \Omega_k.$$
  In particular $f_k(\left(K \cup X \times (0, T_1) \right) \cap N)$ is contained in $M_k \backslash D_k$.
  Because it is connected, it is contained in a unique component $W$ of $\overline{ M_k \backslash D_k }$. We define $\varphi_k(Z) = W$.

  Now we suppose $Z \in V_{\epsilon}^{\omega}(M,g)$. Then there exists some component $L \times (0, \infty)$ of $X \times (0, \infty)$, such that $Z \subset L \times (0, \infty)$. On the other hand
  $$Z \supset L \times (0, \infty) \cap \{v \leq \lambda_0 \epsilon\}.$$
  We also have $X \times (0, T_5) \subset \Omega$ by the choice of $\eta$. By the inequality (\ref{IneqVolSeqVsLim}), $f_k(L \times (T_3, T_5)) \subset D_k$. Since $L \times (T_3, T_5)$ is connected, there exists a unique component $W$ of $D_k$ containing $f_k(L \times (T_3, T_5))$. We define $\varphi_k(Z) = W$.

  {\bf $\bm{\varphi_k}$ is surjective:}\\
  To see this we define a right inverse $\psi_k : V_{\epsilon}(M_k, g_k) \to V_{\epsilon}(M,g)$.
  
  Let $W \in V_{\epsilon}(M_k, g_k)$. By definition
  $$V_{\epsilon}(M_k, g_k) = \pi_0(D_k) \cup \{ Z \in \pi_0(\overline{M \backslash D_k}) : v_{\max}(Z) > \lambda_+ \epsilon\}.$$
  There are two scenarios for $W$:
  \begin{itemize}
  \item $v_{\min}(W) \leq \lambda_- \epsilon, \quad v_{\max}(W) \leq \epsilon$,
  \item $v_{\min}(W) \geq \lambda_- \epsilon, \quad v_{\max}(W) > \lambda_+ \epsilon$.
  \end{itemize}
  In the first case we pick a point $x \in \Omega_k$ such that $v_k(x) = v_{\min}(W \cap \{v \geq \lambda_1 \epsilon\})$. In the second case $W \subset \Omega_k$ and we choose a point $x \in W$ with $v_k(x) = v_{\max}(W)$. (Note that $W$ is closed and bounded, and so the maximum is attained.)

  In both cases we define $\psi_k(W)$ to be the unique connected component in $V_{\epsilon}(M,g)$, which contains $f_k^{-1}(x)$. Note that the components in $V_{\epsilon}(M,g)$ cover $M$, so that there is such a component. Moreover, in the first case we have $v(f_k^{-1}(x)) \leq (1 + 2^{-100}) \lambda_- \epsilon < \lambda_0 \epsilon$. Thus $f_k^{-1}(x)$ is in $D$ but not in $\overline{M \backslash D}$. Similarly, in the second case $v(f_k^{-1}(x)) > (1 - 2^{-100}) \lambda_+ \epsilon$ and so $f_k^{-1}(x) \notin D$. Thus, there is precisely one component $Z$ in $V_{\epsilon}(M,g)$, such that $f_k^{-1}(x) \in Z$.

  To see that $\psi_k$ is a right inverse, let $W \in V_{\epsilon}(M_k, g_k)$.

  We start with the case $v_{\min}(W) \leq \lambda_- \epsilon$. We already saw that in this case $f_k^{-1}(x) \in D$. So there is a unique component $L \times (0, \infty) \subset X \times (0, \infty)$, such that $f_k^{-1}(x) \in L \times (0, \infty)$. By definition $\varphi_k(Z)$ is the component of $D_k$ such that $D_k \supset f_k( L \times (T_3, T_5))$. Since $v(f_k^{-1}(x)) \in ((1-2^{-100})\lambda_1 \epsilon, (1+2^{-100}) \lambda_- \epsilon)$, it follows that $f_k^{-1}(x) \in L \times (T_3, T_5)$. Hence $x = f_k(f_k^{-1}(x)) \in f_k(L \times (T_3, T_5) \subset \varphi_k(\psi_k(W))$. This implies $\varphi_k(\psi_k(W)) \cap W \neq \varnothing$ and so $\varphi_k(\psi_k(W)) = W$.

  Now we move on to the case $v_{\max}(W) \geq \lambda_+ \epsilon$. We already saw $v(f_k^{-1}(x)) > (1 - 2^{-100}) \lambda_+ \epsilon$, which implies $f_k^{-1}(x) \in K \cup X \times (0, T_2)$. Let $N$ be the component of $M$, which contains $f_k^{-1}(x)$. Then by definition $\psi_k(W)$ is the component of $D$, which contains $f_k^{-1}(x)$. In particular, $\psi_k(W) \supset \left(K \cup X \times (0, T_1) \right) \cap N$. From the definition of $\varphi_k(\psi_k(W))$ it follows that $\varphi_k(\psi_k(W)) \supset f_k\left(\left(K \cup X \times (0, T_1) \right) \cap N \right)$. And thus $x \in f_k(f_k^{-1}(x)) \in \varphi_k(\psi_k(W))$. Again we conclude $\varphi_k(\psi_k(W)) = W$. This finishes the proof that
  $$\varphi_k \circ \psi_k = \id : V_{\epsilon}(M_k, g_k) \to V_{\epsilon}(M_k, g_k).$$

  {\bf $\bm{\varphi_k|_{V^{\alpha}_{\epsilon}(M,g)}}$ is injective:}\\
  We show that $\psi_k$ is a left inverse on $V^{\alpha}_{\epsilon}(M,g)$. Suppose $Z \in V^{\alpha}_{\epsilon}(M,g)$.

  Then by definition $\varphi_k(Z) \supset f_k\left( K \cap Z \right)$. We know moreover that $v_{\max}(Z \cap K) \geq \lambda_+ \epsilon$ and thus $v_{\max}(f_k(Z \cap K)) \geq (1 - 2^{-100}) \lambda_+ \epsilon$. For the definition of $\psi_k(\varphi_k(Z))$, we chose $x \in \varphi_k(Z)$ with $v_k(x) = v_{\max}(\varphi_k(Z))$. Thus $v_k(x) \geq (1 - 2^{-100}) \lambda_+ \epsilon$ and thus $v(f_k^{-1}(x)) \geq (1 - 2^{-100})^2 \lambda_+ \epsilon > \epsilon$. We conclude that $\psi_k(\varphi_k(Z)) \in V_{\epsilon}^{\alpha}(M,g)$. Moreover, $x \in f_k(K)$. Since $K \cap Z$ is connected, it follows in fact that $x \in f_k(K \cap Z)$. Using that $f_k^{-1}(x) \in \psi_k(\varphi_k(Z))$, it follows that
  $$Z \cap \psi_k(\varphi_k(Z)) \supset \{f_k^{-1}(x)\} \neq \varnothing.$$
  Since $Z, \psi_k(\varphi_k(Z))$ are components, it follows that $Z = \psi_k(\varphi_k(Z))$ as claimed.

  {\bf $\bm{\varphi_k}$ is a graph morphism:}\\
  Any edge in $\Gamma_{\epsilon}(M,g)$ runs between an element of $V_{\epsilon}^{\alpha}(M,g)$ and an element of $V_{\epsilon}^{\omega}(M,g)$, such that both are in the same component of $M$.

  Thus suppose that there is an edge between $Z_1 \in V_{\epsilon}^{\alpha}(M,g)$ and $Z_2 \in V_{\epsilon}^{\omega}(M,g)$. Let $N$ be the component of $M$ containing $Z_1$ and $Z_2$.

  We need to prove that $\varphi_k(Z_1)$ and $\varphi_k(Z_2)$ are connected by an edge. By definition, this means we have to show that whenever $W \in V_{\epsilon}(M_k, g_k)$ and $W \neq \varphi_k(Z_1), \varphi_k(Z_2)$, then
  $$\delta_{g_k}(\varphi_k(Z_1), \varphi_k(Z_2)) < \delta_{g_k}(\varphi_k(Z_1), W) + \delta_{g_k}(W, \varphi_k(Z_2)).$$
  We have seen that $\varphi_k$ is surjective. Thus we can always assume $W = \varphi_k(Z_3)$ for some $Z_3 \in V_{\epsilon}(M_k,g_k)$.

  Let $N$ be the component of $M$ which contains $Z_1$ and $Z_2$.
  
  We distinguish two cases:
  \begin{itemize}
  \item $Z_3$ is in the same component as $Z_1$ and $Z_2$,
  \item $Z_3$ is in a different component than $Z_2$ and $Z_3$.
  \end{itemize}
  In the first case $Z_3 \subset N$. Since $Z_1$ is the unique element in $V^{\alpha}_{\epsilon}(M,g)$ contained in $N$ it follows that $Z_3 \in V^{\omega}_{\epsilon}(M,g)$. Let $L_1, L_3$ be the components of $X$, such that $Z_2 \subset L_2 \times (T_2, \infty)$ and $Z_3 \subset L_3 \times (T_2, \infty)$. Then
  $$\delta_{g_k}(\varphi_k(Z_1), \varphi_k(Z_2)) < \delta_{g_k}(\varphi_k(Z_1), \varphi_k(Z_3)).$$
  Indeed since any curve from $Z_1$ to $Z_3$ has to pass through $Z_2$ and since $X \times \{T_1\} \subset Z_1$, it follows that
  $$\delta_{g_k} (\varphi_k(Z_2), \varphi_k(Z_3)) \geq \delta_{g_k}(\varphi_k(Z_2), \varphi_k(Z_1)) +  \delta_{g_k}(f_k(L_2 \times \{T_1\}), f_k(L_3 \times \{ T_1\})) + \delta_{g_k}(\varphi_k(Z_1), \varphi_k(Z_3)).$$
  Since $f_k(L_2 \times \{T_1\})$ and $f_k(L_3 \times \{ T_1\})$ are compact and disjoint, it follows that
  $$\delta_{g_k}(f_k(L_2 \times \{T_1\}), f_k(L_3 \times \{ T_1\})) > 0$$
  and thus
  $$\delta_{g_k} (\varphi_k(Z_2), \varphi_k(Z_3)) > \delta_{g_k} (\varphi_k(Z_1), \varphi_k(Z_2)).$$

  In the second case, suppose that $Z_3$ is contained in the component $\tilde{N}$ of $M$. If $Z_3 \in V_{\epsilon}^{\omega}(M,g)$, then we can argue as in the previous case. Thus we may assume that $Z_3$ is the unique element of $V_{\epsilon}^{\alpha}(M,g)$ contained in $\tilde{N}$. Moreover, we let $L$ be the component of $X$, such that $Z_2 \subset L \times (0, \infty)$.

  By definition $K \cup X \times (0,T_5) \subset \Omega$. Any curve between $f_k(K \cap N)$ and $f_k(K \cap \tilde{N})$ in $M_k$ can not wholly lie in $f_k(K \cup X \times (0, T_5))$, i.e.\@ such a curve must pass through $X \times \{T_5\}$.

  Note that
  $$\varphi_k(Z_1) \subset f_k\left( [K \cup X \times (0, T_3)] \cap N \right),\quad \varphi_k(Z_3) \subset f_k\left( [K \cup X \times (0, T_3)] \cap \widetilde{N} \right).$$

  With this in mind we compute
  \begin{align*}
    \delta_{g_k}( \varphi_k(Z_1), \varphi_k(Z_3)) & \geq  \delta_{g_k} \left(f_k\left( [K \cup X \times (0, T_3)] \cap N \right), f_k\left( [K \cup X \times (0, T_3)] \cap \tilde{N} \right)\right) \\
    & \geq \min_{\hat{L} \in \pi_0(X)} \delta_{g_k} \left( f_k(X \times \{T_3\}), f_k(X \times \{T_5\}) \right) \\
    & \stackrel{(\ref{ineq:almostisometry})}{>} \frac{1}{2} \min_{\hat{L} \in \pi_0(X)} \delta_g \left( X \times \{T_3\}, X \times \{T_5\} \right) \\
    & \stackrel{(\ref{ineq:distanceassumption})}{>} 2 \delta_g(L \times \{T_1\}, L \times \{T_3\}) \\
    & \stackrel{(\ref{ineq:almostisometry})}{>} \delta_{g_k}(f_k(L \times \{T_1\}), f_k(L \times \{T_3\})) \\
    & \geq \delta_{g_k}(\varphi_k(Z_1), \varphi_k(Z_2)),
  \end{align*}
  where the last inequality uses that $f_k(L \times \{T_1\}) \subset \varphi_k(Z_1)$ and $f_k(L \times \{T_4\}) \subset Z_2$. Reading only the first and the last terms in this chain of inequalities, we have
  $$\delta_{g_k}(\varphi_k(Z_1), \varphi_k(Z_2)) < \delta_{g_k}(\varphi_k(Z_1), \varphi_k(Z_3)),$$
  which is what we wanted to show.

\end{proof}

\section{The structure of limit spaces of Riemannian surfaces}
\label{SecLimitSurfaces}

This section is concerned with the proof of theorem \ref{ThmCptSurfIntC}. This is achieved through a combination of theorem \ref{ThmRegSubconv} and results of Shioya.

\begin{ethm}[Lemma 3.2,3.3 in \cite{s}]
  \label{ThmBddCurvGH}
  For every $\Lambda > 0$ the family of complete Riemannian surfaces $(M,g)$ with
  $$\int_M |K_g| \vol_g \leq \Lambda$$
  is precompact with respect to the pointed Gromov--Hausdorff topology.
\end{ethm}
Shioya proves this for the class of complete Riemannian surfaces $(M,g)$ with
$$\int_M |K_g| \vol_g \leq \Lambda \quad \text{and} \quad \diam(M,g) \leq D$$
in lemma 3.2 of \cite{s}, but notes in lemma 3.3 that the same technique applies to prove the theorem above for pointed Riemannian surfaces.

Shioya goes on to study the topology of the limit space and shows they have the structure of a pearl space.
\begin{defn}
  A {\em string of pearls} is a topological space $P$ obtained by the following procedure. Let $I$ be a countable index set and suppose $x_i \in (0,1)$ and $\delta_i > 0$, such that the sets $(x_i - \delta_i, x + \delta_i)$ are pairwise disjoint. Let $L = (0,1) \backslash \cup_i (x_i - \delta_i, x + \delta_i)$.
  Denote by $S(p, r) \subset \R^3$ the sphere with center $p \in \R^3$ and radius $r > 0$.

  Then
  $$P =  L \times \{(0,0)\} \cup \bigcup_{i \in I} S((x_i, 0,0), \delta_i).$$
  
  A {\em pearl space} is a topological space $X$ for which every point $x \in X$ admits a neighborhood $U$, such that $U \backslash \{x \}$ is homeomorphic to a disjoint union of a finite number of strings of pearls. The {\em index} of the point $x$ is the number of strings in $U \backslash \{ x\}$.
\end{defn}

\begin{ethm}[Thm. 1.2, 1.4 in \cite{s}]
  \label{ThmTopologicalStructure}
  Let $K, D > 0$.
  
  If $(M_k, g_k)$ is a sequence of complete Riemannian surfaces with $\diam(M_k,g_k) \leq D$ and
  $$\int_{M_k} |K_{g_k}| \vol_{g_k} \leq K$$
  converging in the Gromov--Hausdorff topology to a limit space $(X,d)$, then $(X,d)$ is a pearl space.

  If additionally
  $$\int_{M_k} |K_{g_k}|^p \vol_{g_k} \leq \Lambda$$
  for some $p>1$ and $\Lambda > 0$, then the index of every point is at most $2$.
\end{ethm}
As Shioya remarks in Remark 1.1 (4), the proof of this is local in character, so that the same conclusions hold for limits of metric balls $B(x_k,R)$ in complete surfaces where $\int_{B(x_k, 2R)} |K_{g_k}| \vol_{g_k} \leq K$.

\begin{proof}[Proof of theorem \ref{ThmCptSurfIntC}]
  We recall the assumptions of the theorem: $V, \Lambda > 0$, $p >1$ and $\alpha \in (0, 2 - 2/p)$ are constants and $(M_k, g_k, x_k)$ is a sequence of complete Riemannian surfaces satisfying $\Vol_{g_k}(M_k) \leq V$ and
  $$\left( \int_{M_k} |K_{g_k}|^p \vol_{g_k} \right)^{1/p} \leq \Lambda.$$
  
  By theorem \ref{ThmBddCurvGH} we may pass to a subsequence converging in the pointed Gromov--Hausdorff topology. And because $\Vol_{g_k}(M_k) \leq V$, we may in addition assume $(M_k, d_{g_k}, \Vol_{g_k}, x_k)$ converges in the measured Gromov--Hausdorff topology. This takes care of the first statement in the theorem.

  The second statement in the theorem is that $X^{reg}$ posesses the structure of a $C^{\alpha}$ Riemannian surface and the third theorem is that if $x  \in X^{reg}$, then $(M_k, g_k, x_k)$ converges to $(X^{reg}, g, x)$ in the pointed, volume exhausted $C^{\alpha}$ Cheeger--Gromov topology. These statements are the conclusions of theorem \ref{ThmRegSubconv}.

  For the fourth and fifth statement, we have to study how the topology of the pearl space $(X, d, \mu, x)$ interacts with the metric properties of the sequence $(M_k, g_k, \Vol_{g_k}, x_k)$.

  Let us define
  $$X_2 = \{x \in X : x \text{ has a neighborhood homeomorphic to } \R^2\},$$
  $$X_1 = \{x \in X : x \text{ has a neighboorhood homeomorphic to } \R\}$$
  and
  $$X_0 = \{x \in X : x \text{ is an isolated point}\}.$$
  We claim that $X^{reg} = X_2$ and $X^- = X_1 \cup X_0$.

  Except for the identity $X^- = \{x \in X : \mu(B(x,r)) = 0 \text{ for some } r > 0\}$, this claim proves the remaining assertions of the theorem. Indeed, $X_1$ is a length space, which is homeomorphic to a $1$-manifold. Such manifolds are locally isometric to line segments. The space $X_0$ is metrically a collection of points. Moreover, $X_1$ and $X_0$ are defined by open condition. This is the fourth claim. The fifth claim is that $X^{\perp} = \partial X^{reg} = \partial X^-$ is a discrete subset of $X$ and that $X$ is the disjoint union of $X^{reg}, X^-$ and $X^{\perp}$. Since $X^{reg} = X_2$ and $X^- = X_0 \cup X_1$ are open, we have $X^{\perp} \cap X^{reg} = \varnothing$ and $X^{\perp} \cap X^- = \varnothing$. Moreover, from the definition of a pearl space, we see that every point in a pearl space is either in $X_0, X_1, X_2$ or a limit point of one of these sets. Moreover, the set of these exceptional points is discrete. This proves the fifth claim.

  To see that $X_2 = X^{reg}$, we first note that $X^{reg} \subset X_2$ is immediate, because $X^{reg}$ has the structure of a topological manifold. The proof of $X_2 \subset X^{reg}$ is more involved and requires a detour into metric geometry. The proof of theorem \ref{ThmTopologicalStructure} shows that at any $x \in X$ and for any $\delta > 0$, there exists a $(2,\delta)$-strainer. (Cf. Lemma 6.1 in \cite{s}.) A $(2, \delta)$-strainer at $x$ consists of four points $(p_1, q_1, p_2, q_2)$, which satisfy certain angle conditions depending on $\delta$ with respect to $x$. Existence of such a strainer implies that there exists $\rho > 0$, $\epsilon > 0$, which depend on $\delta$ and the distance of $x$ to the four points, such that
  $$f : B(x, \rho) \subset X \to U \subset \R^2$$
  $$x \mapsto (d(x, q_1), d(x, q_2))$$
  is a $(1+\epsilon)$-bi-Lipschitz homeomorphism onto an open subset $U$ of $\R^2$. Here $\epsilon$ is small, if $\delta$ is small. Since $(X_k, d_k, x_k)$ converges to $(X,d, x)$ in the pointed Gromov--Hausdorff topology, there exists a $(2, \delta_k)$-strainer at $x_k$ for sufficiently large $k$, where $\delta_k$ converges to $\delta$ and the strainers $(p^k_1, q^k_1, p^k_2, q^k_2)$ converge to the strainer $(p_1, q_1, p_2, q_2)$.

  Applying this to the sequence $(M_k, g_k, x_k)$ and supposing $x \in X_2$, we see that there is a $\rho$ and $\epsilon$ independent of $k$, such that
  $$f_k: B(x_k, \rho) \subset M_k \to U_k \subset \R^2$$
  $$x \mapsto (d_{g_k}(x, q^k_1), d_{g_k}(x, q^k_2))$$
  is a $(1+\epsilon)$-bi-Lipschitz homeomorphism. From the definition of a strainer $f_k(U_k) \supset B(0,R_k),$ where $R_k = C(\epsilon) \min \{ \rho, d(x_k, q^k_1), d(x_k, q^k_2), d(x_k, p^k_1), d(x_k, p^k_2)\}$. Because of the convergence of the strainers, it follows that there is $R$ independent of $k$, such that $R \geq R_k$.

  It follows that $\Vol_{g_k} (B(x_k,r)) \geq (1- \alpha(\epsilon)) \omega_n r^n$ for $r < R$, where $\alpha(\epsilon)$ is some constant, which is small when $\epsilon$ is small. In particular, it follows that there is some $\epsilon_0 > 0$, such that $\nu(x_k) > \epsilon_0$. Recall that lemma \ref{LemmaVCFLimit} says that$\nu(x) \geq \limsup_{k \to \infty} \nu(x_k)$. This implies that $\nu(x) \geq \epsilon_0 > 0$ and thus $x \in X^{reg}$.

  Finally, we need to see that $X^- = X_0 \cup X_1$. Suppose $z \in X^-$. Then by definition of $X^{-}$, $z \notin X^{reg} = X_2$. Hence $z \in X_0 \cup X_1 \cup X^{\perp}$. We thus only need to show that $z \notin X^{\perp}$. Suppose $z \in X^{\perp}$. Since $X^{\perp} = \partial X_2$, it follows that for any $r > 0$, the metric ball $B(z,r)$ intersects $X_2$. Hence there exists $y \in X_2$ and $\rho > 0$, such that $B(y,\rho) \subset X_2 \cap B(z,r)$. Since $X_2 = X^{reg}$, it follows that $\mu(B(y,\rho)) > 0$ and hence $\mu(B(z, r)) > 0$. This shows $z \notin X^{\perp}$ and we conclude that $X^- = X_0 \cup X_1$.
\end{proof}


\begin{thebibliography}{99}

\bibitem{a1}
  {\scshape Michael T. Anderson},
  {\em Degeneration of metrics with bounded curvature and applications to critical metrics of Riemannian functionals},
  Proc. Sympos Pure Math., 54, Part 3, AMS, Providence, RI, 1993, 53--77.

\bibitem{a2}
  {\scshape Michael T. Anderson},
  {\em Extrema of curvature functionals on the space of metrics on 3-manifolds},
  Calc. Var. Partial Differential Equations 5 (1997), no. 3, 199--269.
  
\bibitem{pw}
  {\scshape Peter Petersen and Guofang Wei},
  {\em Relative Volume Comparison with Integral Curvature Bounds},
  Geom. Funct. Anal. 7 (1997), no. 6, 1031--1045.


\bibitem{s}
  {\scshape Takashi Shioya},
  {\em The limit spaces of two-dimensional manifolds with uniformly bounded integral curvature},
  Trans. Amer. Math. Soc. 351 (1999), no. 5, 1765--1801.
  
\bibitem{bp}
  {\scshape Christopher Bavard and Pierre Pansu},
  {\em Sur l'espace des surfaces à courbure et aire bornées},
  Ann. Inst. Fourier (Grenoble) 38 (1988), no. 1, 175--203.

\bibitem{c}
  {\scshape Xiuxiong Chen},
  {\em Weak limits of Riemannian metrics in surfaces with integral curvature bound},
  Calc. Var. Partial Differential Equations 6 (1998), no. 3, 189--226.
  
\bibitem{ch}
  {\scshape Eugenio Calabi and Philip Hartman},
  {\em On the smoothness of isometries},
  {Duke Math. J. 37 (1970), 741--750.}
  
\bibitem{d}
  {\scshape Clément Debin},
  {\em A Compactness Theorem for Surfaces with Bounded Integral Curvature},
  J. Inst. Math. Jussieu (2018), 1 -- 49.

\bibitem{eg}
  {\scshape Lawrence C. Evans and Ronald F. Gariepy},
  {Measure theory and fine properties of functions},
  Revised edition. Textbooks in Mathematics. CRC Press, Boca Raton, FL, 2015. 
  
\bibitem{hh}
  {\scshape Emmanuel Hebey and Marc Herzlich},
  {\em Harmonic coordinates, harmonic radius and convergence of Riemannian manifolds},
  Rend. Mat. Appl. (7) 17 (1997), no. 4, 569 -- 605.

  
\bibitem{p1}
  {\scshape Peter Petersen},
  {\em Convergence theorems in Riemannian geometry},
  Math. Sci. Res. Inst. Publ., 30, Cambridge Univ. Press, Cambridge, 1997, 167--202.


\bibitem{p2}
  {\scshape Peter Petersen},
  {\em Riemannian Geometry. (Third Edition)},
  Springer (2016)

\bibitem{t}
  {\scshape Michael Taylor},
  {\em Existence and regularity of isometries},
  {Trans. Amer. Math. Soc. 358 (2006), no. 6, 2415--2423.}



\bibitem{y1}
  {\scshape Deane Yang},
  {\em Convergence of Riemannian manifolds with integral bounds on curvature. I},
  Ann. Sci. École Norm. Sup. (4) 25 (1992), no. 1, 77--105.

  
\bibitem{y3}
  {\scshape Deane Yang},
  {\em Existence and regularity of energy-minimizing Riemannian metrics},
  Internat. Math. Res. Notices 1991, no. 2, 7--13.

\bibitem{y4}
  {\scshape Deane Yang},
  {\em Riemannian manifolds with small integral norm of curvature},
  Duke Math. J. 65 (1992), no. 3, 501--510.


\end{thebibliography}
\end{document}